\documentclass[final]{siamart251216}
\usepackage{titlesec}

\pdfoutput=1
\usepackage{soul}
\usepackage{enumitem}
\usepackage{lipsum}
\usepackage{amsfonts}
\usepackage{epstopdf}
\usepackage{multirow}
\usepackage[colorinlistoftodos,prependcaption,textsize=tiny]{todonotes}
\ifpdf
  \DeclareGraphicsExtensions{.eps,.pdf,.png,.jpg}
\else
  \DeclareGraphicsExtensions{.eps}
\fi

\usepackage{datetime}
\newdateformat{monthyeardate}{%
  \monthname[\THEMONTH] \THEDAY, \THEYEAR}

\makeatletter
\newcommand*{\addFileDependency}[1]{
  \typeout{(#1)}
  \@addtofilelist{#1}
  \IfFileExists{#1}{}{\typeout{No file #1.}}
}
\makeatother

\usepackage{academicons}
\usepackage{xcolor}

\usepackage{amsmath}
\allowdisplaybreaks
\usepackage{amssymb}
\usepackage{commath}
\usepackage{mathtools}
\usepackage{bbm}

\usepackage{color}
\usepackage{graphicx}
\usepackage[small]{caption}
\usepackage{subcaption}

\usepackage{relsize}
\usepackage{adjustbox}
\usepackage{algorithm,algpseudocode}
\usepackage{booktabs}
\usepackage{tikz}

\usepackage{verbatim}

\usepackage{enumitem}
\setlist[enumerate]{leftmargin=.5in}
\setlist[itemize]{leftmargin=.5in}

\newsiamthm{problem}{Problem}
\newsiamthm{assumption}{Assumption}
\newsiamremark{remark}{Remark}
\newsiamremark{hypothesis}{Hypothesis} 
\crefname{hypothesis}{Hypothesis}{Hypotheses}
\newsiamremark{example}{Example}
\newsiamthm{claim}{Claim}
\newsiamthm{conjecture}{Conjecture}

\headers{NL-RMM-GKS for Dynamic and Streaming Inverse Problems}{
Okunola T., Pasha M., Kilmer M., Nagy J. de Sturler E.
}
\title{Nonlinear RMM-GKS for Large-Scale Dynamic and Streaming Inverse Problems with Uncertain Forward Operators
}
\author{
Toluwani Okunola\thanks{Department of Mathematics, Tufts University} (\email{toluwani.okunola@tufts.edu})
\and 
Mirjeta Pasha\thanks{Department of Mathematics, Virginia Tech} (\email{mpasha@vt.edu})
\and
Misha E. Kilmer\thanks{Department of Mathematics, Tufts University}    (\email{misha.kilmer@tufts.edu})
\and
James G. Nagy \thanks{Institute of Mathematics, Emory University} (\email{jnagy@emory.edu})
\and
Eric de Sturler\thanks{Department of Mathematics, Virginia Tech} (\email{sturler@vt.edu})}
\usepackage{amsopn}

\usepackage{ifthen}
\usepackage{float}
\floatstyle{ruled}
\newfloat{algorithm}{htb}{alg}
\floatname{algorithm}{Algorithm}
\newcounter{algo@row}
\newcounter{algo@rowindent}
\newcommand{\algofont}[1]{\textbf{#1}}
\newcommand{\algonumbersize}[1]{\scriptsize{#1}}
\newcommand{\algopreitem}[1][\arabic{algo@row}]{\texttt{\algonumbersize{#1}}}
\newcommand{\algoitemskip}{\hspace{\value{algo@rowindent}cc}}
\newcommand{\algonewnestedopen}[2]{
	\newcommand{#1}[1][]{%
		\ifthenelse{\equal{##1}{}}{\item}{\item[{\algopreitem[##1]}]}
		\algoitemskip\algofont{#2}%
		\addtocounter{algo@rowindent}{1}%
		\ignorespaces
	}
}
\newcommand{\algonewnestedaux}[2]{
	\newcommand{#1}[1][]{
		\addtocounter{algo@rowindent}{-1}
		\ifthenelse{\equal{##1}{}}{\item}{\item[{\algopreitem[##1]}]}
		\algoitemskip\algofont{#2}%
		\addtocounter{algo@rowindent}{+1}%
		\ignorespaces
	}
}
\newcommand{\algonewnestedclose}[2]{
	\newcommand{#1}[1][]{
		\addtocounter{algo@rowindent}{-1}
		\ifthenelse{\equal{##1}{}}{\item}{\item[{\algopreitem[##1]}]}
		\algoitemskip\algofont{#2}%
		\ignorespaces
	}
}
\newcommand{\algonewcommand}[2]{
	\newcommand{#1}[1][default]{
		\ifthenelse{\equal{##1}{default}}{\item}{\item[{\algopreitem[##1]}]}%
		\algoitemskip\algofont{#2}%
		\ignorespaces
	}%
}
\newcommand{\algonewkeyword}[2]{\newcommand{#1}{\algofont{#2}}}
\algonewcommand{\STATE}{\ignorespaces}
\algonewcommand{\INPUT}{Input: }
\algonewcommand{\pINPUT}{\phantom{Input: }}
\algonewcommand{\COMPUTE}{Compute: }
\algonewcommand{\OUTPUT}{Output: }
\algonewcommand{\pOUTPUT}{\phantom{Output: }}

\algonewnestedopen{\IF}{if }
\algonewnestedaux{\ELSEIF}{else if }
\algonewnestedaux{\ELSE}{else }
\algonewnestedclose{\ENDIF}{end if }
\algonewnestedopen{\FOR}{for }
\algonewnestedclose{\ENDFOR}{end for }
\algonewnestedopen{\WHILE}{while }
\algonewnestedclose{\ENDWHILE}{end while }
\algonewcommand{\BREAK}{break}%

\algonewkeyword{\To}{to }%
\algonewkeyword{\Do}{do }%
\algonewkeyword{\Then}{then }%
\algonewkeyword{\End}{end }%
\algonewkeyword{\AND}{and }%
\algonewkeyword{\True}{true }%
\algonewkeyword{\False}{false }%
\algonewkeyword{\irbleigs}{irbleigs }%
\algonewkeyword{\tridiag}{tridiag}%
\algonewkeyword{\reorth}{reorth}%

\usepackage{cleveref}
\crefname{section}{§}{§§} 

\DeclareMathOperator*{\argmin}{arg\,min}

\newcommand{\bV}{{\bf V}}

\usepackage{textcomp}
\usepackage{gensymb}

\usepackage{amsmath}
\usepackage{amssymb}
\usepackage[utf8]{inputenc}

\usepackage{algcompatible}
\usepackage{algorithm}
\usepackage{url}

\graphicspath{{figures/}}

\begin{document}
\maketitle

\begin{abstract}
Many practical imaging systems suffer from uncertainty in acquisition geometry---such 
as projection angles in computed tomography or sensor positions in photoacoustic 
tomography---leading to nonlinear inverse problems that require joint estimation of 
both the image and the forward model parameters. Standard approaches that assume a 
known linear forward operator fail to account for these uncertainties, resulting in 
significant reconstruction artifacts.
We propose a nonlinear recycled majorization–minimization generalized Krylov subspace (NL-RMM-GKS) framework for large-scale inverse problems with uncertain forward operators. The method extends MM-GKS to nonlinear settings by combining majorization–minimization for nonsmooth regularization with Krylov subspace projection and subspace recycling, ensuring bounded memory usage.

Two complementary formulations are developed: an alternating minimization approach that alternates between image updates and Gauss–Newton parameter estimation, and a variable projection approach that eliminates the image variable and optimizes directly over the parameters using inexact inner solves. We further introduce streaming variants that process data sequentially, enabling reconstruction from large or dynamically acquired datasets without storing the full operator. For dynamic problems, we incorporate two temporal regularization strategies---optical flow and anisotropic total variation---as plug-in choices within the framework.
We carry out rigorous numerical experiments in fan-beam computed tomography and photoacoustic tomography to demonstrate that our proposed framework achieves high-quality reconstructions with bounded memory requirements, making it suitable for large-scale dynamic imaging problems.
\end{abstract}



\section{Introduction.}
\label{sec:intro}
Inverse problems arise throughout science and engineering when one seeks to determine 
unknown parameters or images from indirect measurements. In medical imaging, deblurring, 
geophysical exploration, and astronomical observation, the challenge is that of 
reconstructing an unknown image $\mathbf{u} \in \mathbb{R}^n$ from noisy measurements 
$\mathbf{b} \in \mathbb{R}^m$ related through a forward operator. In many practical 
scenarios, the forward operator depends on unknown parameters $\mathbf{p} \in 
\mathbb{R}^{n_p}$, leading to a nonlinear relationship:
\begin{equation}
\label{eq:nonlinear_forward}
\mathbf{H}(\mathbf{p})\mathbf{u} + \mathbf{e} = \mathbf{b},
\end{equation}
where $\mathbf{H}(\mathbf{p}) \in \mathbb{R}^{m \times n}$ represents the forward 
operator, $\mathbf{e}$ denotes measurement noise, and $\mathbf{p}$ encodes geometric 
or calibration uncertainties. Consider computed tomography where projection angles or 
source locations may be imperfectly calibrated, photoacoustic imaging where sensor 
positions contain uncertainties, or deblurring problems where the point spread function 
is not precisely known. This problem is often underdetermined or ill-posed, making the 
least-squares solution non-unique or highly sensitive to noise. To combat these issues, 
regularization is employed, leading to the joint estimation problem:
\begin{equation}
\label{eq:nonlinear_model}
\min_{\mathbf{u},\, \mathbf{p}} \frac{1}{2}\|\mathbf{H}(\mathbf{p})\mathbf{u} - 
\mathbf{b}\|_2^2 + \lambda\|\boldsymbol{\Psi}\mathbf{u}\|_q^q,
\end{equation}
where $0 < q \leq 2$, $\boldsymbol{\Psi}$ represents a regularization operator, and 
$\lambda > 0$ is a regularization parameter. We are thus presented with the joint tasks 
of estimating both the image $\mathbf{u}$ and the geometric or calibration parameters 
$\mathbf{p}$. This joint estimation problem is significantly more challenging than the 
linear case, as it couples the typically large-dimensional image reconstruction problem 
with a smaller but highly nonlinear parameter estimation problem. The challenge becomes 
even more pronounced in \emph{dynamic} imaging settings, where the object itself 
changes over time---as in cardiac CT or time-lapse photoacoustic imaging---and the 
forward operator may simultaneously depend on unknown geometric parameters. In such 
cases, temporal regularization must be incorporated alongside joint image-parameter
estimation, further coupling an already challenging reconstruction problem.
{Alternative approaches include Bayesian formulations with spatiotemporal priors that promote edge   
  preservation through wavelet-based representations~\cite{lan2025spatiotemporal}, and sparse Bayesian learning approaches that embed sparsity-promoting priors into Krylov subspace solvers~\cite{lindbloom2025priorconditioned}.}

In the case of image reconstruction, $q$ in Equation~\eqref{eq:nonlinear_model} is 
often set to $1$ in order to ensure that the regularization term is a valid 
norm~\cite{buccini2021linearized, chung2022, chung2019flexible}. The effect of this 
choice depends on $\boldsymbol{\Psi}$: when $\boldsymbol{\Psi}$ is a gradient or finite difference operator, 
$q = 1$ promotes edge preservation (akin to total variation regularization), whereas 
when $\boldsymbol{\Psi}$ is the identity, it promotes sparsity in the solution. We will therefore 
use $q = 1$ throughout the paper, but the framework we present can easily be 
generalized to other values of $q$.

In the special case where the forward operator is fully known ($\mathbf{p}$ fixed), 
\eqref{eq:nonlinear_forward} reduces to the classical linear model 
$\mathbf{H}\mathbf{u} + \mathbf{e} = \mathbf{b}$, and \eqref{eq:nonlinear_model} 
reduces to the standard regularized least-squares problem:
\begin{equation}
\label{eq:mmgks_problem}
\min_{\mathbf{u}} \frac{1}{2}\|\mathbf{H}\mathbf{u} - \mathbf{b}\|_2^2 + 
\lambda\|\boldsymbol{\Psi}\mathbf{u}\|_q^q,
\end{equation}
which has been extensively studied in the
literature~\cite{buccini2021linearized, chung2022, chung2019flexible}.

\subsection{Existing Approaches and Their Limitations.}
The challenge of reconstructing images when the forward operator contains unknown parameters has been addressed through several distinct approaches.
Block coordinate descent (BCD) methods, also known as alternating minimization, alternate between fixing one variable set and optimizing over the other, converting the joint minimization into a sequence of simpler subproblems \cite{bonettini2011inexact, tseng2009coordinate}. However, BCD is sensitive to initialization and tends to converge to local minima \cite{meng2023numerical}.

Variable projection (VarPro), introduced by Golub and Pereyra \cite{golub1973differentiation, golub2003separable}, offers an elegant alternative by eliminating the image variable analytically. Given a separable $\ell_2$ problem where measurements depend linearly on the image $\mathbf{u}$ and nonlinearly on parameters $\mathbf{p}$, VarPro substitutes $\mathbf{u} = \mathbf{H}(\mathbf{p})^\dagger \mathbf{b}$ to obtain a reduced objective depending only on $\mathbf{p}$. It has been shown in the literature (see \cite{kaufman1975varpro, ruhe1974seperable, espanol2023variable}) that the VarPro approach of separating the linear variable $\mathbf{u}$ from the nonlinear variable $\mathbf{p}$ speeds up the convergence of iterative methods used to solve \eqref{eq:nonlinear_model} in the case $q=2$. The case $1 \le q<2$ which we consider here is more challenging, as there is no closed-form expression for the solution $\mathbf{u}(\mathbf{p})$ of the problem for fixed $\mathbf{p}$ (see \cite{espanol2021ell_p}).

More recently, machine learning and optimization-hybrid methods have emerged for joint calibration and reconstruction in nonlinear inverse problems. Here we use the term ``hybrid'' to mean a combination of learned and iterative optimization components, distinct from the hybrid Krylov methods used for linear regularization \cite{chung2022}. Machine learning models have been trained to predict appropriate regularization parameters directly \cite{afkham2021learning, chung2022learning}. Inspired by this, Meng et al.\ \cite{meng2023numerical} introduce a framework that calibrates geometry parameters using a neural network, and then refines the predicted parameters and reconstructs the imaged object simultaneously using BCD. This two-stage strategy aims to combine the speed and generalization of learned models with the reliability and interpretability of optimization-based refinement. However, such hybrid methods inherit the limitations of their constituent parts---they still rely on BCD for the refinement stage with its associated convergence issues, and the learned components require substantial training data and may not generalize well to out-of-distribution scenarios or geometries not seen during training. Furthermore, the learning phase introduces an additional layer of complexity that can be difficult to implement in practice.

For the linear case ($\mathbf{p}$ known), the majorization-minimization generalized Krylov subspace (MM-GKS) method introduced by Lanza et al.\ \cite{mm-gks} has proven highly effective for solving large-scale inverse problems \cite{buccini2020modulus, pasha2023computational}. The algorithm addresses the non-smooth $\ell_1$-type regularization by constructing a sequence of quadratic majorants, each of which can be minimized efficiently using projection onto a small generalized Krylov subspace. MM-GKS tackles problems of the form \eqref{eq:mmgks_problem} by iteratively refining a weighted quadratic approximation to the objective, automatically selecting the regularization parameter and solving the quadratic problem on small projected subspaces.

However, MM-GKS in its original form has two significant limitations for large-scale and nonlinear problems. First, the Krylov subspace dimension grows with each iteration, requiring storage of an expanding basis. For problems requiring hundreds of iterations to converge, this becomes prohibitively costly. Second, the method assumes a fixed, known forward operator $\mathbf{H}$, and cannot directly handle the nonlinear dependence on parameters $\mathbf{p}$ in equation \eqref{eq:nonlinear_model}.

Recent developments have partially addressed the memory limitation. In \cite{buccini2023limited}, a restarted MM-GKS algorithm is introduced which periodically restarts the algorithm to reduce memory and computational costs. However, after restarting, their approach ignores the previous solution subspace, discarding relevant information that has been collected. The recycled MM-GKS (RMM-GKS) method introduced by Pasha et al.\ in \cite{pasha2023recycling} alternates between enlarging the Krylov subspace and compressing it to retain only the most informative directions. This keeps the memory bounded while maintaining convergence properties. The authors also introduce a streaming variant (s-RMM-GKS) that further extends this to sequential data processing, accommodating scenarios where measurements arrive in blocks or where the full dataset cannot be stored due to limited memory. For dynamic problems with known geometry, the MM-GKS framework has been 
combined with optical flow regularization (MMGKS-OF) \cite{okunola2025mmgks-of}, 
which alternates between estimating forward and reverse optical flow velocity fields 
and reconstructing the image sequence, incorporating the estimated motion as a 
linearized regularization operator within each MM-GKS solve. A unified framework 
for edge-preserving dynamic reconstruction using spatio-temporal anisotropic total 
variation regularization---which simultaneously promotes sparse spatial gradients 
per frame and penalizes inter-frame differences via a joint $\ell_1$ penalty---was 
developed in \cite{pasha2023computational}.

Despite these advances, however, the fundamental challenge of nonlinear geometry remains unaddressed within the MM-GKS framework.
When the forward operator depends on unknown parameters, existing methods either treat the parameters as fixed, accepting the resulting artifacts, or alternate between image reconstruction and parameter estimation using generic nonlinear solvers that discard valuable Krylov subspace information after each outer iteration. None of these approaches fully leverage the structure of the problem or the efficiency of Krylov-based methods, which we address in the present work.

\subsection{Contributions.}
In this paper, we develop a comprehensive framework for nonlinear inverse problems with uncertain forward operators that extends and unifies previous MM-GKS variants. Our contributions include:

\paragraph{Nonlinear RMM-GKS framework (NL-RMM-GKS).} We present two complementary realizations for joint image-parameter estimation:
\begin{itemize}
\item An \textbf{alternating minimization (AltMin)} approach that majorizes the non-smooth penalty and alternates between updating the image via recycled Krylov projections and updating parameters through Gauss-Newton steps.
\item A \textbf{variable projection (VarPro)} approach that eliminates the image variable through an inner MM-GKS solve and performs Gauss-Newton optimization on the reduced objective.
\end{itemize}
Both realizations share the same inner MM-GKS machinery, regularization strategies, and parameter selection methods, differing only in how they handle parameter updates.

\paragraph{Recycling for nonlinear problems.} We adapt the enlarge-compress recycling strategy to the nonlinear setting, maintaining bounded memory while preserving essential search directions. The recycling mechanism intelligently retains information about the solution structure even as the forward operator changes with parameter updates.

\paragraph{Streaming extensions (s-NL-RMM-GKS).} We develop streaming variants that process measurement blocks sequentially, carrying over basis information between blocks. This enables reconstruction from data that arrives sequentially or from datasets too large to fit in memory.

\paragraph{Temporal regularization for dynamic problems.} For dynamic imaging sequences, we incorporate two plug-in temporal regularization strategies within the NL-RMM-GKS framework: optical flow regularization \cite{okunola2025mmgks-of}, which enforces a physics-based motion model between frames, and anisotropic total variation \cite{pasha2023recycling}, which penalizes frame-to-frame differences without assuming a specific motion structure.

\paragraph{Unified methodology.} Our framework naturally generalizes MMGKS, RMM-GKS, s-RMM-GKS, and MMGKS-OF as special cases, providing a unified formulation of these methods.

\subsection{Organization.}
The remainder of this paper is organized as follows. In Section~\ref{sec:preliminaries} 
we establish notation and review essential background on inverse problems and 
regularization. In Section~\ref{sec:background} we summarize the MM-GKS method and 
its recycling and streaming variants for linear problems. Our main contributions are 
presented in Section~\ref{sec:nonlinear}: the nonlinear RMM-GKS framework with both 
AltMin and VarPro realizations, convergence analysis, and detailed algorithms. 
Streaming extensions of NL-RMM-GKS are developed in Section~\ref{sec:streaming}. 
Section~\ref{sec:experiments} contains static numerical experiments in computed 
tomography and photoacoustic tomography. The extension to dynamic problems, including 
the temporal regularization strategies and dynamic experiments, is presented in 
Section~\ref{sec:dynamic}. We conclude in Section~\ref{sec:conclusions} with a 
summary and discussion of future directions. A proof of the main convergence result 
is provided in Appendix~\ref{sec:appendix_proof}.

\section{Preliminaries.}
\label{sec:preliminaries}
We start by introducing our notation, the regularization used in the specification 
of the problem, and the temporal regularization strategies employed for dynamic 
imaging.

\subsection{Notation and Problem Setup.}
We denote vectors by lowercase bold letters ($\mathbf{u}, \mathbf{b}, \mathbf{p}$) 
and matrices by uppercase bold letters ($\mathbf{H}, \boldsymbol{\Psi}$). For 
dynamic problems involving multiple time frames, we stack the images as 
$\mathbf{u} = \text{vec}([\mathbf{u}^{(1)}, \ldots, \mathbf{u}^{(n_t)}])$, where 
each $\mathbf{u}^{(j)} \in \mathbb{R}^{n_s}$ represents a single frame and 
$n = n_s n_t$ is the total dimension.

The forward operator $\mathbf{H}(\mathbf{p}) \in \mathbb{R}^{m \times n}$ maps 
the image space to the measurement space and depends on parameters 
$\mathbf{p} \in \mathbb{R}^{n_p}$, which encode geometric or calibration 
information such as perturbations in projection angles or detector positions in 
computed tomography, radial shifts in photoacoustic imaging, or point spread 
function parameters in deblurring. The measurement vector $\mathbf{b} \in 
\mathbb{R}^m$ contains noisy observations related to the true image through 
equation~\eqref{eq:nonlinear_forward}.

\subsection{Regularization Operator.}
For static problems, $\boldsymbol{\Psi}$ represents a finite difference 
approximation to the gradient in horizontal and vertical directions. For dynamic 
problems, we include temporal differences:
\begin{equation}
\boldsymbol{\Psi} = \begin{bmatrix} \mathbf{I}_{n_t} \otimes \mathbf{I}_{n_y} 
\otimes \mathbf{L}_x \\ \mathbf{I}_{n_t} \otimes \mathbf{L}_y \otimes 
\mathbf{I}_{n_x} \\ \mathbf{L}_t \otimes \mathbf{I}_{n_y} \otimes \mathbf{I}_{n_x} 
\end{bmatrix},
\end{equation}
where $\mathbf{L}_x, \mathbf{L}_y, \mathbf{L}_t$ are discrete derivative operators 
in the horizontal, vertical, and temporal directions respectively, and $\otimes$ 
denotes the Kronecker product. The choice of how the temporal component 
$\mathbf{L}_t$ is constructed determines the temporal regularization strategy, 
as described in Section~\ref{sec:temporal_reg} below.

\subsection{Temporal Regularization for Dynamic Problems.}
\label{sec:temporal_reg}

For dynamic imaging problems, the choice of temporal regularization strategy 
significantly affects reconstruction quality. We consider two plug-in strategies 
within the NL-RMM-GKS framework: anisotropic space-time total variation 
(ANISO-TV)~\cite{pasha2023computational}, which promotes sparse spatial gradients 
and penalizes frame-to-frame intensity differences, and optical flow 
regularization (OF)~\cite{okunola2025mmgks-of}, which enforces a physics-based 
motion model between frames. Both strategies are described in full detail in 
Section~\ref{sec:dynamic}, where their mathematical formulations and algorithms 
are presented alongside the dynamic numerical experiments.

\subsection{Streaming Data and Block Processing.}
\label{sec:streaming_setup}
For streaming applications, we partition the forward operator and measurement 
vector into $N$ blocks:
\begin{equation}
\mathbf{H}(\mathbf{p}) = \begin{bmatrix} \mathbf{H}_1(\mathbf{p}) \\ \vdots \\ 
\mathbf{H}_N(\mathbf{p}) \end{bmatrix}, \quad \mathbf{b} = \begin{bmatrix} 
\mathbf{b}_1 \\ \vdots \\ \mathbf{b}_N \end{bmatrix},
\end{equation}
where each $\mathbf{H}_j(\mathbf{p}) \in \mathbb{R}^{m_j \times n}$ and 
$\mathbf{b}_j \in \mathbb{R}^{m_j}$ corresponds to a randomly selected subset 
of measurements (e.g., a randomly assigned block of projection angles in CT). 
The measurements are partitioned by drawing without replacement, so that each 
measurement appears in exactly one block. The streaming algorithms 
process block $j$ in sequence, updating the solution and parameter estimates 
while carrying over the compressed Krylov basis from block $j-1$ to block $j$. 
The number of blocks $N$ controls the memory-quality trade-off: larger $N$ 
reduces peak memory and runtime but may degrade reconstruction quality, a 
trade-off we explore systematically in Section~\ref{sec:experiments}.

\section{Background: MM-GKS and Related Methods.}
\label{sec:background}

Here we review the MM-GKS framework for linear inverse problems 
\cite{mm-gks, lanza2015generalized} and its extensions to recycling and 
streaming from \cite{pasha2023recycling}. These methods form the foundation 
for our nonlinear algorithms.

\subsection{Majorization-Minimization Approach.}

The $\ell_1$ regularization in equation \eqref{eq:mmgks_problem} is 
non-differentiable at zero, making standard gradient-based optimization 
challenging. The majorization-minimization (MM) approach addresses this by 
constructing a sequence of smooth quadratic approximations to
\[\mathcal{J}(\mathbf{u}) := \frac{1}{2}\|\mathbf{H}\mathbf{u} - 
\mathbf{b}\|_2^2 + \lambda\|\boldsymbol{\Psi}\mathbf{u}\|_1.\]

Given a current iterate $\mathbf{u}^{(k)}$, we define 
$\mathbf{z}^{(k)} = \boldsymbol{\Psi}\mathbf{u}^{(k)}$ and construct 
diagonal weights:
\begin{equation}
\mathbf{P}_\epsilon^{(k)} = \mathrm{diag}\left({\sqrt{
(\mathbf{z}^{(k)}_i)^2 + \epsilon^2}}\right)^{-1/2},
\end{equation}
where $\epsilon > 0$ ensures differentiability. These weights give rise to 
a quadratic majorant:
\begin{equation}
\label{eq:quadratic_majorant}
\mathcal{Q}(\mathbf{u} \mid \mathbf{u}^{(k)}) = \frac{1}{2}\|\mathbf{H}
\mathbf{u} - \mathbf{b}\|_2^2 + \frac{\lambda}{2}\|\mathbf{P}_\epsilon^{(k)}
\boldsymbol{\Psi}\mathbf{u}\|_2^2 + c,
\end{equation}
where $c$ is a constant independent of $\mathbf{u}$. This majorant satisfies:
(i) $\mathcal{Q}(\mathbf{u}^{(k)} \mid \mathbf{u}^{(k)}) = 
\mathcal{J}(\mathbf{u}^{(k)})$,
(ii) $\nabla\mathcal{Q}(\mathbf{u}^{(k)} \mid \mathbf{u}^{(k)}) = 
\nabla\mathcal{J}(\mathbf{u}^{(k)})$, and
(iii) $\mathcal{Q}(\mathbf{u} \mid \mathbf{u}^{(k)}) \geq \mathcal{J}(\mathbf{u})$ 
for all $\mathbf{u}$.
Minimizing the majorant leads to the normal equations:
\begin{equation}
\label{eq:normal_equations}
(\mathbf{H}^\top\mathbf{H} + \lambda\boldsymbol{\Psi}^\top(\mathbf{P}_\epsilon^{(k)})^2
\boldsymbol{\Psi})\mathbf{u} = \mathbf{H}^\top\mathbf{b}.
\end{equation}
The MM principle guarantees that minimizing the majorant decreases the original 
objective, ensuring convergence to a stationary point under mild conditions.

\subsection{Generalized Krylov Subspace Projection.}

Solving \eqref{eq:normal_equations} directly is impractical for large-scale 
problems. MM-GKS instead projects onto a small subspace spanned by generalized 
Krylov vectors that capture information about both the data misfit and the 
regularization. Starting from an initial basis $\mathbf{V}_\ell \in 
\mathbb{R}^{n \times \ell}$ with orthonormal columns (generated by Golub-Kahan 
bidiagonalization applied to $\mathbf{H}$), we seek an approximate solution 
$\mathbf{u}^{(k+1)} = \mathbf{V}_\ell\mathbf{z}$ by computing QR factorizations
\begin{equation}
\mathbf{H}\mathbf{V}_\ell = \mathbf{Q}_A\mathbf{R}_A, \quad 
\mathbf{P}_\epsilon^{(k)}\boldsymbol{\Psi}\mathbf{V}_\ell = \mathbf{Q}_\Psi\mathbf{R}_\Psi,
\end{equation}
and solving the reduced least-squares problem:
\begin{equation}
\label{eq:reduced_problem}
\mathbf{z}^{(k+1)} = \arg\min_{\mathbf{z} \in \mathbb{R}^\ell} 
\left\|\begin{bmatrix} \mathbf{R}_A \\ \sqrt{\lambda}\mathbf{R}_\Psi 
\end{bmatrix}\mathbf{z} - \begin{bmatrix} \mathbf{Q}_A^\top\mathbf{b} \\ 
\mathbf{0} \end{bmatrix}\right\|_2^2.
\end{equation}
The regularization parameter $\lambda$ is selected at each iteration using 
the discrepancy principle applied to the reduced problem (generalized cross validation (GCV) may be used 
alternatively when the noise level is unknown). If the residual of the 
normal equations is not sufficiently small, the subspace is expanded by 
appending the normalized residual direction:
\begin{equation}
\mathbf{r}^{(k+1)} = \mathbf{H}^\top(\mathbf{H}\mathbf{u}^{(k+1)} - 
\mathbf{b}) + \lambda\boldsymbol{\Psi}^\top(\mathbf{P}_\epsilon^{(k)})^2
\boldsymbol{\Psi}\mathbf{u}^{(k+1)},
\end{equation}
\begin{equation}
\mathbf{V}_{\ell+1} = [\mathbf{V}_\ell,\; \mathbf{r}^{(k+1)}/
\|\mathbf{r}^{(k+1)}\|_2].
\end{equation}
This expansion is repeated, updating the weights $\mathbf{P}_\epsilon^{(k)}$ 
at each iteration, until convergence. The full MM-GKS algorithm is given in 
Algorithm~\ref{alg:mmgks} of the supplementary material.

\subsection{Recycling: RMM-GKS.}
\label{sec:rmmgks}

For large-scale problems requiring many iterations, MM-GKS exhausts available 
memory as the subspace dimension grows without bound. The RMM-GKS algorithm 
\cite{pasha2023recycling} addresses this through an enlarge-compress cycle that 
keeps memory bounded while preserving the most informative search directions.

\paragraph{Enlargement.} Starting from a basis of dimension $k_{\min}$, 
MM-GKS expansion steps are performed as described above, growing the subspace 
up to a maximum dimension $k_{\max}$.

\paragraph{Compression.} When the basis reaches $k_{\max}$ columns, it is 
compressed back to $k_{\min}$ columns by applying a compression function 
$\chi$ that computes the truncated SVD \footnote{Because $H_{k_{max}}$ is relatively small, computing the tSVD is inexpensive.} of the stacked matrix
\begin{equation}
\mathbf{H}_{k_{\max}} :=
\begin{bmatrix} \mathbf{R}_A \\ \sqrt{\lambda_{\mathrm{curr}}}\mathbf{R}_\Psi 
\end{bmatrix},
\end{equation}
and retains the $k_{\min}-1$ right singular vectors corresponding to the 
largest singular values, forming $\mathbf{W} \in \mathbb{R}^{k_{\max} \times 
(k_{\min}-1)}$. The compressed basis is $\tilde{\mathbf{V}} = \mathbf{V}_{k_{\max}}
\mathbf{W}$. To guarantee monotonic decrease of the objective, the normalized 
component of the current solution orthogonal to $\tilde{\mathbf{V}}$ is appended, 
forming the new $k_{\min}$-dimensional basis. This enlarge-compress cycle repeats 
until convergence. The Enlarge and Compress subroutines are given as 
Algorithms~\ref{alg:enlarge} and~\ref{alg:compress} in the supplementary material; 
the full RMM-GKS algorithm is given as Algorithm~\ref{alg:rmmgks} below.

\begin{algorithm}
\caption{RMM-GKS (with optional fixed $\boldsymbol{\lambda}$)}
\label{alg:rmmgks}
\begin{algorithmic}[1]
\REQUIRE $\mathbf{H},\boldsymbol{\Psi},\mathbf{d},\mathbf{u}^{(0)},\bV^{(0)},
k_{\min},k_{\max},\epsilon,\text{tol}_1,\lambda_{\mathrm{fix}}$
\ENSURE $(\mathbf{u}^{(i+1)},\mathbf{V}_{k_{\min}},\lambda^{(i)})$

\STATE $s = k_{\max}-k_{\min}$
\IF{$\bV^{(0)}$ is given}
    \STATE $\mathbf{V}_{k_{\min}}$ = $\bV^{(0)}$
\ELSE
    \STATE Initialize basis $\mathbf{V}_{k_{\min}}$ via Golub-Kahan 
    bidiagonalization
\ENDIF

\FOR{$i=1,2,\ldots,i_{\max}$}
    \STATE $(\mathbf{u}^{(i+1)},\lambda^{(i+1)},\mathbf{V}_{k_{\max}},
    R_H,R_\Psi) = \mathrm{Enlarge}(
        \mathbf{H},\boldsymbol{\Psi},
        \mathbf{V}_{k_{\min}+1},
        \mathbf{d},\mathbf{u}^{(i)},
        \epsilon,s,\text{tol}_1,
        \lambda_{\mathrm{fix}})$

    \STATE $\mathbf{V}_{k_{\min}} = \mathrm{Compress}(
        \mathbf{V}_{k_{\max}},R_H,R_\Psi,
        \mathbf{d},\mathbf{u}^{(i+1)},Q_H,
        k_{\min},\lambda^{(i+1)})$

    \STATE Update weights $\mathbf{P}^{(i+1)}_\epsilon$ from $\mathbf{u}^{(i+1)}$
    \STATE $\mathbf{r}^{(i+1)} = \mathbf{H}^\top(\mathbf{H}\mathbf{u}^{(i+1)} 
    - \mathbf{b}) + \lambda^{(i+1)}\boldsymbol{\Psi}^\top
    (\mathbf{P}^{(i+1)}_\epsilon)^2\boldsymbol{\Psi}\mathbf{u}^{(i+1)}$
    \STATE $\mathbf{r}^{(i+1)} \leftarrow \mathbf{r}^{(i+1)} - 
    \mathbf{V}_{k_{\min}}\mathbf{V}_{k_{\min}}^\top\mathbf{r}^{(i+1)}$
    \STATE $\mathbf{V}_{k_{\min}+1} \leftarrow [\mathbf{V}_{k_{\min}}, 
    \mathbf{r}^{(i+1)}/\|\mathbf{r}^{(i+1)}\|_2]$
\ENDFOR
\STATE \textbf{Output:} $(\mathbf{u}^{(i+1)},\mathbf{V}_{k_{\min}},
\lambda^{(i+1)})$
\end{algorithmic}
\end{algorithm}

\subsection{Streaming: s-RMM-GKS.}
\label{sec:srmmgks}

The streaming variant of RMM-GKS \cite{pasha2023recycling} addresses scenarios 
where data arrives sequentially or memory constraints prevent processing the full 
dataset simultaneously. Given blocks $\{\mathbf{H}_j, \mathbf{b}_j\}_{j=1}^N$ 
as defined in Section~\ref{sec:streaming_setup}, we process each block in 
sequence, initializing with the solution $\mathbf{u}^{(j-1)}$ and compressed 
basis $\mathbf{V}^{(j-1)}_{k_{\min}}$ carried over from the previous block. 
This allows reconstruction from arbitrarily large datasets with constant memory 
requirements, as formalized in Algorithm~\ref{alg:s_rmmgks}.

\begin{algorithm}
\caption{s-RMM-GKS}
\label{alg:s_rmmgks}
\begin{algorithmic}[1]
\REQUIRE Blocks $\{(\mathbf{H}_j, \mathbf{d}_j)\}_{j=1}^N$, 
regularizer $\boldsymbol{\Psi}$, initial iterate $\mathbf{u}^{(0)}$, 
basis size $(k_{\min}, k_{\max})$, parameter $\epsilon$, 
tolerance $\text{tol}_1$
\ENSURE Approximate streaming solution $\mathbf{u}^*$

\STATE Initialize $\mathbf{V}^{(0)}_{k_{\min}}$

\FOR{$j = 1$ to $N$}
    \STATE $(\mathbf{u}^{(j)},\, \mathbf{V}^{(j)}_{k_{\min}})
    = \mathrm{RMM\mbox{-}GKS}\!
    \left(
        \mathbf{H}_j,\,\boldsymbol{\Psi},\,\mathbf{d}_j,\,
        \mathbf{u}^{(j-1)},\,\mathbf{V}^{(j-1)}_{k_{\min}},\,
        k_{\min},\,k_{\max},\,\epsilon,\,\text{tol}_1
    \right)$
\ENDFOR
\STATE $\mathbf{u}^* = \mathbf{u}^{(N)}$
\end{algorithmic}
\end{algorithm}

\section{Nonlinear RMM-GKS Framework.}
\label{sec:nonlinear}

We now extend RMM-GKS to handle forward operators $\mathbf{H}(\mathbf{p})$ 
that depend nonlinearly on unknown parameters $\mathbf{p}$. We present two 
complementary realizations of the joint optimization problem
\begin{equation}
\label{eq:nonlinear_joint}
\min_{\mathbf{u}, \mathbf{p}} \frac{1}{2}\|\mathbf{H}(\mathbf{p})\mathbf{u} 
- \mathbf{b}\|_2^2 + \lambda\|\boldsymbol{\Psi}\mathbf{u}\|_1:
\end{equation}
\begin{enumerate}
    \item \textbf{Alternating Minimization (AltMin):} Decouples the joint 
    problem by alternating between image updates via RMM-GKS and parameter 
    updates via Gauss-Newton steps.
    \item \textbf{Variable Projection (VarPro):} Eliminates the image variable 
    through an implicit function, reducing to optimization over parameters only.
\end{enumerate}
Both approaches share the same inner MM-GKS machinery and differ only in how 
parameter updates are coupled to image reconstruction. Problem~\eqref{eq:nonlinear_joint} 
is nonlinear in $\mathbf{p}$ and non-smooth in $\mathbf{u}$; the former requires 
Gauss-Newton iteration while the latter is handled by the majorization-minimization 
approach of Section~\ref{sec:background}.

\subsection{Alternating Minimization (AltMin).}
\label{sec:altmin}

Alternating minimization decouples the joint problem into two 
subproblems solved iteratively, fully optimizing each variable while holding 
the other fixed.

\subsubsection{Image Update.}
With the current parameter estimate $\mathbf{p}^{(k)}$ fixed, we solve for 
the image by constructing the quadratic majorant and minimizing over a recycled 
Krylov subspace:
\begin{equation}
\label{eq:altmin_image}
\mathbf{u}^{(k+1)} = \argmin_{\mathbf{u}} \frac{1}{2}\|\mathbf{H}(\mathbf{p}^{(k)})
\mathbf{u} - \mathbf{b}\|_2^2 + \frac{\lambda}{2}\|\mathbf{P}_\epsilon^{(k)}
\boldsymbol{\Psi}\mathbf{u}\|_2^2,
\end{equation}
where $\mathbf{P}_\epsilon^{(k)}$ is computed from $\mathbf{u}^{(k)}$ as in 
Section~\ref{sec:background}. We apply the enlarge-compress cycle of RMM-GKS, 
projecting onto subspaces of dimension between $k_{\min}$ and $k_{\max}$ and 
selecting $\lambda$ adaptively via the discrepancy principle on the reduced 
problem. Although $\mathbf{H}(\mathbf{p}^{(k)})$ changes across outer iterations 
as parameters are updated, within each inner RMM-GKS solve the operator is fixed, 
so RMM-GKS applies without modification.

\subsubsection{Parameter Update.}
With the updated image $\mathbf{u}^{(k+1)}$ fixed, we minimize the data 
fidelity term with respect to parameters using damped Gauss-Newton iteration 
with backtracking line search:
\begin{equation}
\label{eq:altmin_param}
\mathbf{p}^{(k+1)} = \arg\min_{\mathbf{p}} \frac{1}{2}\|\mathbf{H}(\mathbf{p})
\mathbf{u}^{(k+1)} - \mathbf{b}\|_2^2.
\end{equation}
At each inner iteration $\ell$, the Gauss-Newton direction solves
\begin{equation}
\Delta\mathbf{p}^{(\ell)} = \arg\min_{\Delta\mathbf{p}} \|\mathbf{J}^{(\ell)}
\Delta\mathbf{p} + \mathbf{r}^{(\ell)}\|_2^2,
\end{equation}
where $\mathbf{J}^{(\ell)} = \left.\frac{\partial}{\partial\mathbf{p}}
(\mathbf{H}(\mathbf{p})\mathbf{u}^{(k+1)})\right|_{\mathbf{p}=\mathbf{p}^{(\ell)}}$ 
and $\mathbf{r}^{(\ell)} = \mathbf{H}(\mathbf{p}^{(\ell)})\mathbf{u}^{(k+1)} - 
\mathbf{b}$. The step size $\alpha^{(\ell)} \in (0,1]$ is chosen via backtracking 
line search satisfying the Armijo condition, and inner iterations continue until 
$\|\Delta\mathbf{p}^{(\ell)}\| < \text{tol}_p$ or $\ell = \ell_{\max}^p$. 
The full parameter update procedure is given in Algorithm~\ref{alg:update_param}.

\begin{algorithm}
\caption{UPDATE-PARAM (AltMin parameter update at outer iteration $k$)}
\label{alg:update_param}
\begin{algorithmic}[1]
\REQUIRE $\mathbf{u}^{(k)},\,\mathbf{b},\,\mathbf{p}^{(k-1)} \in 
\mathbb{R}^{n_p}$, $\text{maxiter}_p \geq 1$, damping $\mu > 0$
\ENSURE $\mathbf{p}^{(k)}$
\STATE Initialize: $\mathbf{p}^{(k-1,0)} \leftarrow \mathbf{p}^{(k-1)}$
\FOR{$\ell=0,\ldots,\text{maxiter}_p-1$}
    \STATE $\mathbf{g}^{(k,\ell)} = \nabla_{\mathbf{p}} 
    \mathcal{J}_{\epsilon,\lambda}(\mathbf{u}^{(k)}, \mathbf{p}^{(k-1,\ell)})$
    \FOR{$i=1,\ldots,n_p$}
        \STATE $\mathbf{H}_{p_i} = \frac{\partial\mathbf{H}
        (\mathbf{p}^{(k-1,\ell)})}{\partial p_i}$
        \STATE $[\mathbf{J}_{\mathbf{u}}]_{:,i} = \mathbf{H}_{p_i}\,\mathbf{u}^{(k)}$
    \ENDFOR
    \STATE $\mathbf{J}_{\mathbf{p}}^{(k,\ell)} = \mathbf{J}_{\mathbf{u}}^\top
    \mathbf{J}_{\mathbf{u}} + \mu \mathbf{I}$
    \STATE Solve $\mathbf{J}_{\mathbf{p}}^{(k,\ell)}\,\mathbf{d}^{(k,\ell)} = 
    -\mathbf{g}^{(k,\ell)}$
    \STATE Find $\alpha^{(k,\ell)}$ satisfying Armijo condition
    \STATE $\mathbf{p}^{(k-1,\ell+1)} = \mathbf{p}^{(k-1,\ell)} + 
    \alpha^{(k,\ell)}\,\mathbf{d}^{(k,\ell)}$
\ENDFOR
\STATE $\mathbf{p}^{(k)} \leftarrow \mathbf{p}^{(k-1,\text{maxiter}_p)}$
\end{algorithmic}
\end{algorithm}

\subsection{Variable Projection (VarPro).}
\label{sec:varpro}

The variable projection approach eliminates $\mathbf{u}$ by solving the 
majorized image subproblem implicitly, reducing the joint problem to 
optimization over $\mathbf{p}$ alone.

\paragraph{Majorization and Closed-Form Solution.}
Given an iterate $\mathbf{u}^{(k)}$, we form weights $\mathbf{P}^{(k)} = 
\mathrm{diag}(( (\boldsymbol{\Psi}\mathbf{u}^{(k)})_j^2 + \epsilon^2)^{-1/4})$ 
and define the weight operator $\hat{\boldsymbol{\Psi}}^{(k)} = 
\mathbf{P}^{(k)}\boldsymbol{\Psi}$. For fixed $\mathbf{p}$ and 
$\mathbf{P}^{(k)}$, the majorized subproblem
\begin{equation}
\label{eq:majorized_subproblem}
\min_{\mathbf{u}} \frac{1}{2}\|\mathbf{H}(\mathbf{p})\mathbf{u}-\mathbf{b}\|_2^2
+ \frac{\lambda}{2} \|\hat{\boldsymbol{\Psi}}^{(k)}\mathbf{u}\|_2^2
\end{equation}
has the closed-form solution
\begin{equation}
\label{eq:u_closed_form}
\mathbf{u}(\mathbf{p}) = \left( \mathbf{H}(\mathbf{p})^\top \mathbf{H}(\mathbf{p}) 
+ \lambda (\hat{\boldsymbol{\Psi}}^{(k)})^\top\hat{\boldsymbol{\Psi}}^{(k)} 
\right)^{-1} \mathbf{H}(\mathbf{p})^\top \mathbf{b}.
\end{equation}

\paragraph{Reduced Objective and Gauss-Newton Update.}
Substituting \eqref{eq:u_closed_form} into the objective defines the reduced 
function $f(\mathbf{p}) = \frac{1}{2}\|\mathbf{r}(\mathbf{p})\|_2^2 + 
\frac{\lambda}{2}\|\mathbf{s}(\mathbf{p})\|_2^2$, where $\mathbf{r}(\mathbf{p}) 
= \mathbf{H}(\mathbf{p})\mathbf{u}(\mathbf{p}) - \mathbf{b}$ and 
$\mathbf{s}(\mathbf{p}) = \hat{\boldsymbol{\Psi}}^{(k)}\mathbf{u}(\mathbf{p})$. 
The Gauss-Newton update solves
\begin{equation}
\label{eq:GN_system}
\left( \mathbf{J}_{\mathbf{r}}^\top \mathbf{J}_{\mathbf{r}}
+ \lambda\,\mathbf{J}_{\mathbf{s}}^\top \mathbf{J}_{\mathbf{s}} \right) 
\Delta\mathbf{p} = -\left( \mathbf{J}_{\mathbf{r}}^\top \mathbf{r}
+ \lambda\,\mathbf{J}_{\mathbf{s}}^\top \mathbf{s} \right),
\end{equation}
where the Jacobian columns $[\mathbf{J}_{\mathbf{r}}]_{(:,j)}$ and 
$[\mathbf{J}_{\mathbf{s}}]_{(:,j)}$ are computed via the sensitivity equation
\begin{equation}
\label{eq:sensitivity}
\mathbf{M}_\lambda(\mathbf{p})\,\mathbf{u}_{p_j} = -\left( 
\mathbf{H}_{p_j}(\mathbf{p})^\top \mathbf{r}(\mathbf{p}) + 
\mathbf{H}(\mathbf{p})^\top \mathbf{H}_{p_j}(\mathbf{p})\,\mathbf{u}(\mathbf{p}) 
\right),
\end{equation}
with $\mathbf{M}_\lambda(\mathbf{p}) = \mathbf{H}(\mathbf{p})^\top\mathbf{H}(\mathbf{p}) 
+ \lambda(\hat{\boldsymbol{\Psi}}^{(k)})^\top\hat{\boldsymbol{\Psi}}^{(k)}$ and 
$\mathbf{H}_{p_j} = \partial\mathbf{H}/\partial p_j$. In practice the inner 
subproblem~\eqref{eq:majorized_subproblem} is solved inexactly via RMM-GKS, 
and the weights $\mathbf{P}^{(k)}$ are updated each outer iteration. The 
VarPro parameter update procedure is given in Algorithm~\ref{alg:update_varpro} 
of the supplementary material.

\subsection{High-Regularization Stabilization.}
\label{sec:hireg}

For both AltMin and VarPro, when the adaptively chosen $\lambda^{(k)}$ is 
small, the image update $\mathbf{u}^{(k+1)}$ may be insufficiently regularized, 
causing the subsequent parameter update to be attracted to poor local minima. 
To mitigate this, we optionally perform a second RMM-GKS solve with a 
significantly larger regularization parameter $\lambda_{\mathrm{hi}} \gg 
\lambda^{(k)}$, producing a smoother image estimate $\hat{\mathbf{u}}^{(k+1)}$ 
that is used \emph{only} for the parameter update. The primary solution 
$\mathbf{u}^{(k+1)}$ (obtained with the adaptively chosen $\lambda^{(k)}$) 
is retained as the image reconstruction output and its associated compressed 
basis $\mathbf{V}^{(k+1)}_{k_{\min}}$ is carried forward to the next outer 
iteration. The high-regularization solve uses the basis output of the primary 
solve as its starting point but produces a separate compressed basis that is 
discarded after the parameter update. This stabilization strategy is summarized 
in Algorithm~\ref{alg:nl_rmmgks_two_solves}.

\begin{algorithm}
\caption{NL-RMM-GKS}
\label{alg:nl_rmmgks_two_solves}
\begin{algorithmic}[1]
\REQUIRE $\mathbf{b},\boldsymbol{\Psi},\mathbf{p}^{(0)},k_{\min},k_{\max},\epsilon$
\ENSURE $(\mathbf{u}^*,\mathbf{p}^*)$
\STATE Initialize $\mathbf{V}^{(0)}_{k_{\min}}$
\FOR{$k = 0,1,2,\ldots$}
    \STATE $\mathbf{H}_k = \mathbf{H}(\mathbf{p}^{(k)})$
    \STATE \textbf{Primary image solve:}
    \[
    (\mathbf{u}^{(k+1)},\lambda^{(k)},\mathbf{V}^{(k+1)}_{k_{\min}}) = 
    \mathrm{RMM\text{-}GKS}(
    \mathbf{H}_k,\boldsymbol{\Psi}, \mathbf{b},\mathbf{u}^{(k)}, 
    \mathbf{V}^{(k)}_{k_{\min}}, k_{\min},k_{\max},\epsilon)
    \]
    \STATE \textbf{Optional high-regularization solve:}
    \[
    (\hat{\mathbf{u}}^{(k+1)},\hat{\lambda}^{(k)},\hat{\mathbf{V}}_{k_{\min}}) = 
    \mathrm{RMM\text{-}GKS}(
    \mathbf{H}_k,\boldsymbol{\Psi}, \mathbf{b},\mathbf{u}^{(k+1)}, 
    \mathbf{V}^{(k+1)}_{k_{\min}}, k_{\min},k_{\max},\epsilon;
    \lambda_{\mathrm{hi}})
    \]
    \hfill // $\hat{\mathbf{V}}_{k_{\min}}$ is discarded after this step
    \STATE \textbf{Parameter update:}
    \[
    \mathbf{p}^{(k+1)} = \mathrm{UPDATE\text{-}PARAM}(
    \mathbf{H}_k,\boldsymbol{\Psi}, \hat{\mathbf{u}}^{(k+1)},\mathbf{b}, 
    \mathbf{p}^{(k)},\hat{\lambda}^{(k)})
    \]
    \STATE \textbf{Update weights:} $\mathbf{P}^{(k+1)}_\epsilon$ from 
    $\mathbf{u}^{(k+1)}$
    \STATE \textbf{Compute and project residual:}
    \STATE $\mathbf{r}^{(k+1)} = \mathbf{H}_k^\top(\mathbf{H}_k
    \mathbf{u}^{(k+1)} - \mathbf{b}) + \lambda^{(k)}
    \boldsymbol{\Psi}^\top(\mathbf{P}^{(k+1)}_\epsilon)^2
    \boldsymbol{\Psi}\mathbf{u}^{(k+1)}$
    \STATE $\mathbf{r}^{(k+1)} \leftarrow \mathbf{r}^{(k+1)} - 
    \mathbf{V}^{(k+1)}_{k_{\min}}\left(\mathbf{V}^{(k+1)}_{k_{\min}}\right)^\top
    \mathbf{r}^{(k+1)}$
    \STATE \textbf{Update search space:}
    \STATE $\mathbf{V}^{(k+1)}_{k_{\min}} \leftarrow 
    [\mathbf{V}^{(k+1)}_{k_{\min}},\; 
    \mathbf{r}^{(k+1)}/\|\mathbf{r}^{(k+1)}\|_2]$
    \STATE Check outer convergence
\ENDFOR
\STATE $\mathbf{u}^*=\mathbf{u}^{(k+1)},\;\mathbf{p}^*=\mathbf{p}^{(k+1)}$
\end{algorithmic}
\end{algorithm}

\subsection{Convergence Analysis.}
\label{sec:convergence}

The following assumptions are standard in nonlinear optimization and hold 
for common imaging operators such as CT with angular perturbations and PAT 
with radial shifts.

\begin{assumption}[Regularity conditions]
\label{assump:regularity}
(i) The forward operator $\mathbf{H}(\mathbf{p})$ is twice continuously 
differentiable with bounded derivatives: $\|\partial\mathbf{H}/\partial p_i\|_2 
\leq C_H$ and $\|\partial^2\mathbf{H}/\partial p_i\partial p_j\|_2 \leq C_{HH}$. 
(ii)The iterates are bounded: $\|\mathbf{u}^{(k)}\|_2 \leq R_u$ and 
$\|\mathbf{p}^{(k)}\|_2 \leq R_p$. (iii) The damped Gauss-Newton matrix satisfies 
$\gamma_{\min}\mathbf{I} \preceq \mathbf{J}_{\mathbf{p}}^{(k,\ell)} \preceq 
\gamma_{\max}\mathbf{I}$ for constants $0 < \gamma_{\min} \leq \gamma_{\max} 
< \infty$, where $\gamma_{\min} \geq \mu$ is ensured by the damping parameter.
\end{assumption}
\begin{theorem}[Convergence to stationary point]
\label{thm:convergence}
Let $\{(\mathbf{u}^{(k)}, \mathbf{p}^{(k)})\}$ be the sequence generated by 
Algorithm~\ref{alg:nl_rmmgks_two_solves} under Assumption~\ref{assump:regularity}. 
Then every limit point $(\mathbf{u}^*, \mathbf{p}^*)$ is a stationary point 
of~\eqref{eq:nonlinear_joint}, satisfying
\begin{equation}
\nabla_{\mathbf{u}} \mathcal{J}_{\epsilon,\lambda}(\mathbf{u}^*, \mathbf{p}^*) 
= \mathbf{0} \quad \text{and} \quad \nabla_{\mathbf{p}} 
\mathcal{J}_{\epsilon,\lambda}(\mathbf{u}^*, \mathbf{p}^*) = \mathbf{0}.
\end{equation}
\end{theorem}

\begin{remark}
The proof establishes convergence by showing that each outer iteration produces 
sufficient descent in $\mathcal{J}_{\epsilon,\lambda}$: the image update decreases 
the objective by at least $\|\nabla_{\mathbf{u}}\mathcal{J}\|_2^2 / (2\bar{\mu})$, 
where $\bar{\mu} = C_{\mathbf{H}}^2 + \lambda\varepsilon^{-1}C_{\boldsymbol{\Psi}}^2$ 
bounds the majorant Hessian, while the Gauss-Newton parameter update with 
Armijo line search decreases it by at least $c_1\alpha_{\min}\gamma_{\max}^{-1}
\|\nabla_{\mathbf{p}}\mathcal{J}\|_2^2$. Since the objective is bounded below, 
both gradient norms are summable and hence converge to zero. The complete proof, 
including all technical lemmas and the extension to multiple inner iterations, 
is given in Appendix~\ref{sec:appendix_proof}.
\end{remark}

\begin{remark}
Theorem~\ref{thm:convergence} provides a first-order necessary condition for 
optimality. Global convergence to the global minimum cannot be guaranteed for 
nonconvex problems; however, our experiments in Section~\ref{sec:experiments} 
demonstrate that the method reliably recovers ground truth parameters when 
initialized via the coarse grid search described in 
Section~\ref{sec:experiments}.
\end{remark}
\begin{remark}[Relationship to prior work]
When $\mathbf{p}$ is known and fixed, Algorithm~\ref{alg:nl_rmmgks_two_solves} 
reduces exactly to RMM-GKS~\cite{pasha2023recycling} applied to the linear 
problem~\eqref{eq:mmgks_problem}. The dynamic extension presented in 
Section~\ref{sec:dynamic} further reduces to MMGKS-OF~\cite{okunola2025mmgks-of} 
(with optical flow) or the framework of~\cite{pasha2023computational} (with ANISO-TV) 
when $\mathbf{p}$ is known. NL-RMM-GKS therefore strictly generalizes all of 
these methods.
\end{remark}

\begin{remark}[Scope of convergence guarantee]
Theorem~\ref{thm:convergence} and its proof (Appendix~\ref{sec:appendix_proof}) assume that the parameter update operates on the same image $\mathbf{u}^{(k+1)}$ produced by the primary RMM-GKS solve. When the optional high-regularization stabilization (Section~4.3) is active, the parameter update instead uses the smoothed estimate $\hat{\mathbf{u}}^{(k+1)}$, and the descent chain requires an additional argument relating $\mathcal{J}(\hat{\mathbf{u}}^{(k+1)}, \mathbf{p}^{(k+1)})$ to $\mathcal{J}(\mathbf{u}^{(k+1)}, \mathbf{p}^{(k+1)})$. The convergence guarantee therefore applies directly to Algorithm~4.2 when the high-regularization step is omitted; extending it formally to the stabilized variant is left to future work.
\end{remark}

\section{Streaming Extensions.}
\label{sec:streaming}

We now extend both AltMin and VarPro realizations of NL-RMM-GKS to handle 
streaming data, where measurement blocks $\{(\mathbf{H}_j(\cdot), 
\mathbf{b}_j)\}_{j=1}^N$ arrive sequentially or cannot all be held in memory 
simultaneously (see Section~\ref{sec:streaming_setup} for the block 
partitioning setup). The key idea is to carry over three quantities between 
blocks: the current image estimate $\mathbf{u}^{(j)}$, the current parameter 
estimate $\mathbf{p}^{(j)}$, and the compressed Krylov basis 
$\mathbf{V}^{(j)}_{k_{\min}}$. Carrying over the basis is what distinguishes 
streaming NL-RMM-GKS from simply running NL-RMM-GKS independently on each 
block: the basis encodes solution structure learned from all previously seen 
data, allowing each new block to build on rather than discard prior information. 
Multiple passes over the full dataset are performed until the outer convergence 
criterion is met. When $\mathbf{p}$ is known and fixed, 
s-NL-RMM-GKS reduces exactly to s-RMM-GKS \cite{pasha2023recycling}, 
confirming that our framework is a strict generalization of the linear streaming 
method.

\subsection{Algorithm.}

At each pass, blocks $j = 1, \ldots, N$ are processed in sequence. For block 
$j$, we run NL-RMM-GKS (Algorithm~\ref{alg:nl_rmmgks_two_solves}) using only 
the data $(\mathbf{H}_j(\cdot), \mathbf{b}_j)$, initialized with the solution 
triple $(\mathbf{u}^{(j-1)}, \mathbf{p}^{(j-1)}, \mathbf{V}^{(j-1)}_{k_{\min}})$ 
carried over from the previous block. The outputs $(\mathbf{u}^{(j)}, 
\mathbf{p}^{(j)}, \mathbf{V}^{(j)}_{k_{\min}})$ are then passed to block $j+1$. 
At the end of each pass, convergence is checked on the full dataset; if not 
converged, the next pass begins with the final block's outputs used to 
initialize block 1. The full procedure is given in 
Algorithm~\ref{alg:s_nl_rmmgks_final}.

\begin{algorithm}
\caption{s-NL-RMM-GKS (Streaming Nonlinear RMM-GKS)}
\label{alg:s_nl_rmmgks_final}
\begin{algorithmic}[1]
\REQUIRE $\{(\mathbf{H}_j(\cdot),\mathbf{b}_j)\}_{j=1}^N$,
$\boldsymbol{\Psi}$,
$(\mathbf{u}^{(0)},\mathbf{p}^{(0)})$,
$k_{\min},k_{\max},\epsilon,\lambda_{\mathrm{fix}},\lambda_{\mathrm{hi}}$,
$\text{tol}_{\mathrm{outer}}$
\ENSURE $(\mathbf{u}^*,\mathbf{p}^*)$

\STATE Initialize $\mathbf{V}^{(0)}_{k_{\min}}$

\FOR{pass $= 1, 2, \ldots$ until outer convergence}

    \FOR{$j = 1$ \textbf{to} $N$}

        \STATE $\mathbf{H}_j = \mathbf{H}_j(\mathbf{p}^{(j-1)})$

        \STATE \textbf{Primary image solve:}
        \[
            (\mathbf{u}^{(j)},\mathbf{V}^{(j)}_{k_{\min}},\lambda^{(j)})
            = \mathrm{RMM\!-\!GKS}(
                \mathbf{H}_j,\boldsymbol{\Psi},
                \mathbf{b}_j,\mathbf{u}^{(j-1)},
                \mathbf{V}^{(j-1)}_{k_{\min}},
                k_{\min},k_{\max},\epsilon,\text{tol}_u)
        \]

        \STATE \textbf{Optional high-regularization solve:}
        \[
            (\hat{\mathbf{u}}^{(j)},\hat{\mathbf{V}}_{k_{\min}},\hat{\lambda}^{(j)})
            = \mathrm{RMM\!-\!GKS}(
                \mathbf{H}_j,\boldsymbol{\Psi},
                \mathbf{b}_j,\mathbf{u}^{(j)},
                \mathbf{V}^{(j)}_{k_{\min}},
                k_{\min},k_{\max},\epsilon,\text{tol}_u,
                \lambda_{\mathrm{hi}})
        \]
        \hfill // $\hat{\mathbf{V}}_{k_{\min}}$ discarded after this step

        \STATE \textbf{Parameter update:}
        \[
            \mathbf{p}^{(j)}
            = \mathrm{UPDATE\!-\!PARAM}(
                \mathbf{H}_j,\boldsymbol{\Psi},
                \hat{\mathbf{u}}^{(j)},\mathbf{b}_j,
                \mathbf{p}^{(j-1)},\hat{\lambda}^{(j)})
        \]

    \ENDFOR

    \STATE Check outer convergence on full dataset:
    \[
        \frac{\|\mathbf{u}^{(N)} - \mathbf{u}^{(N)}_{\mathrm{prev}}\|}
        {\|\mathbf{u}^{(N)}_{\mathrm{prev}}\|} < \text{tol}_{\mathrm{outer}}
        \quad \text{and} \quad
        \|\Delta\mathbf{p}^{(N)}\| < \text{tol}_p
    \]

\ENDFOR

\STATE $\mathbf{u}^* = \mathbf{u}^{(N)},\quad \mathbf{p}^* = \mathbf{p}^{(N)}$
\end{algorithmic}
\end{algorithm}

\subsection{Memory and Complexity Analysis.}
\label{sec:complexity}

A key advantage of the streaming approach is that memory requirements are 
bounded independently of the total dataset size. We summarize the dominant 
memory costs below.

\paragraph{Basis storage.} At any point in the algorithm, only the current 
compressed basis $\mathbf{V}^{(j)}_{k_{\min}} \in \mathbb{R}^{n \times k_{\min}}$ 
and the enlarged basis $\mathbf{V}^{(j)}_{k_{\max}} \in \mathbb{R}^{n \times 
k_{\max}}$ need to be stored. Since $k_{\min}$ and $k_{\max}$ are fixed 
hyperparameters, basis storage costs $\mathcal{O}(n k_{\max})$ regardless 
of the number of blocks $N$ or the number of passes. By contrast, standard 
MM-GKS without recycling requires $\mathcal{O}(n K)$ storage where $K$ is 
the total iteration count, which grows unboundedly.

\paragraph{Forward operator storage.} Only the current block 
$\mathbf{H}_j(\mathbf{p}^{(j-1)}) \in \mathbb{R}^{m_j \times n}$ needs to 
be formed and stored at any one time, costing $\mathcal{O}(m_j n)$ where 
$m_j = m/N$ for equal-sized blocks. The full operator $\mathbf{H}(\mathbf{p}) 
\in \mathbb{R}^{m \times n}$ is never formed explicitly, giving an 
$\mathcal{O}(N)$ reduction in operator storage compared to the non-streaming 
case.

\paragraph{Parameter and image storage.} The image $\mathbf{u}^{(j)} \in 
\mathbb{R}^n$ and parameter vector $\mathbf{p}^{(j)} \in \mathbb{R}^{n_p}$ 
are carried between blocks at cost $\mathcal{O}(n + n_p)$, which is 
independent of $N$.

\paragraph{Total memory.} The dominant cost is basis storage 
$\mathcal{O}(nk_{\max})$, which is fixed by the recycling window size and 
independent of $N$, the number of passes, and the total number of measurements 
$m$. In practice, the $\mathcal{O}(N)$ reduction in operator storage is the 
more significant saving for large-scale problems, as $m \gg n$ is common in 
tomographic applications. The trade-off between $N$ and reconstruction quality 
is explored empirically in Section~\ref{sec:experiments}.

\paragraph{Per-pass complexity.} Each pass over all $N$ blocks requires 
$N \cdot \mathcal{O}(k_{\max} - k_{\min})$ inner RMM-GKS iterations, each 
costing $\mathcal{O}(m_j n)$ for the matrix-vector products with 
$\mathbf{H}_j$. The Gauss-Newton parameter update costs $\mathcal{O}(n_p^2 
m_j)$ per block for forming and solving the reduced system. Since $n_p \ll 
m_j$ in our applications (we estimate a small number of geometric parameters), 
this cost is negligible relative to the image solve.

\section{Numerical Results.}
\label{sec:experiments}

We demonstrate the effectiveness of NL-RMM-GKS on computed tomography and 
photoacoustic tomography problems with geometric uncertainty. All CT experiments 
are simulated using the TRIPS-Py package \cite{pashapythonpackage} and the 
ASTRA toolbox \cite{astra}, while PAT problems are simulated using a Python 
translation of the IR-Tools package \cite{gazzola2018ir}. Our experiments 
address the following questions:

\begin{enumerate}
\item \textbf{Memory efficiency:} Does recycling match the convergence of 
standard MM-GKS while maintaining bounded memory?

\item \textbf{Streaming performance:} Can streaming variants (s-NL-RMM-GKS) 
achieve near-full-data quality, and how does performance vary with the number 
of blocks $N$?

\item \textbf{Dynamic reconstruction:} For dynamic phantoms, how do temporal 
regularization strategies (ANISO-TV vs.\ optical flow) perform when combined 
with geometry estimation?

\item \textbf{Recycling window size:} How does $k_{\max}$ affect the 
trade-off between memory, convergence speed, and reconstruction quality?
\end{enumerate}

The AltMin vs.\ VarPro robustness comparison is addressed in the supplementary 
material (Section~S4), where we show that VarPro converges faster with good 
initialization while AltMin is more robust to poor initialization.

\subsection{Experimental Setup.}

\paragraph{Quality metrics.}
We evaluate reconstruction quality using the relative reconstruction error 
(RRE) $\|\mathbf{u}^{(k)} - \mathbf{u}_{\text{true}}\|_2 / 
\|\mathbf{u}_{\text{true}}\|_2$ and the relative parameter error 
$\|\mathbf{p}^{(k)} - \mathbf{p}_{\text{true}}\|_2 / 
\|\mathbf{p}_{\text{true}}\|_2$.

\paragraph{Stopping criteria.}
We terminate when the relative solution change satisfies 
$\|\mathbf{u}^{(k+1)} - \mathbf{u}^{(k)}\|/\|\mathbf{u}^{(k)}\| < 10^{-3}$, 
the parameter increment satisfies $\|\Delta\mathbf{p}^{(k)}\| < 10^{-4}$, 
or the maximum iteration count is reached.

\paragraph{Common parameters.}
Unless otherwise stated, all experiments use recycling window $k_{\min} = 5$, 
$k_{\max} = 25$, truncated SVD compression retaining $k_{\min}-1$ singular 
vectors, and regularization parameter selected by the discrepancy principle. 
All experiments are run on a Dell XPS 16 with single-threaded execution.

\paragraph{Parameter initialization.}
We initialize $\mathbf{p}^{(0)}$ using a coarse grid search in a 
low-dimensional subspace: an initial basis $\mathbf{V}_{\ell_0}$ with 
$\ell_0 = 10$ is generated via Golub-Kahan bidiagonalization at the nominal 
geometry, and for each candidate $\mathbf{p}_{\text{cand}}$ on a coarse grid 
we solve the reduced problem with $\mathbf{H}(\mathbf{p}_{\text{cand}})
\mathbf{V}_{\ell_0}$ and select the candidate yielding the smallest data 
residual. This is computationally cheap since only the reduced subspace problem 
is solved.


\subsection{Computed Tomography with Uncertain Projection Angles.}

In fan-beam CT, the forward operator $\mathbf{H}(\mathbf{p})$ implements the 
Radon transform at perturbed angles $\boldsymbol{\theta}(p) = [\theta_1^{\text{nom}} 
+ p, \ldots, \theta_{n_\theta}^{\text{nom}} + p]^\top$, where the scalar $p$ 
represents an unknown global angular shift. Even small errors of $\pm 1$--$2^\circ$ 
can produce significant reconstruction artifacts \cite{meng2023numerical}, motivating 
joint estimation of $p$ alongside the attenuation image.

\subsubsection{Test 1: Static Shepp-Logan Reconstruction.}
\label{sec:test1}

\paragraph{Setup.}
We use a $256 \times 256$ Shepp-Logan phantom with $n_\theta = 180$ projection 
angles randomly selected from $[0^\circ, 180^\circ)$ and $n_r = 362$ detector bins. 
Geometry uncertainty is introduced by a random angular shift $p_{\text{true}} 
= 0.2421^\circ$, with 1\% Gaussian noise added to the sinogram. We initialize at 
$p^{(0)} = 0^\circ$ via grid search. For streaming experiments, the 180 angles are 
partitioned into $N$ randomly assigned blocks of approximately equal size 
($N = 2, 4, 6, 10$).

\paragraph{Goal.} Our first goal is to compare the computational efficiency, memory usage, and convergence behavior of standard MM-GKS without recycling (memory grows linearly with iterations), NL-RMM-GKS ($N=1$, full data with recycling), and streaming s-NL-RMM-GKS with varying block counts ($N = 2, 4, 6, 10$).
Additionally, in Section~\ref{sec:supp_inner_iters}, we investigate how the number of inner solver iterations affects streaming performance when $N=4$.

\paragraph{Results.}
Figure~\ref{fig:exp_1a_convergence_plot} shows that all methods successfully 
recover the true angle (relative parameter error $< 10^{-3}$), but 
reconstruction quality differs substantially. Standard MM-GKS without recycling 
stagnates at RRE $= 0.1805$, while NL-RMM-GKS ($N=1$) converges to RRE $= 
0.0790$---a 56\% improvement---demonstrating that recycling improves convergence 
quality, not just memory efficiency. Increasing $N$ trades quality for speed 
(Figure~\ref{fig:ct_streaming_blocks}): at $N=10$ we achieve a 5.2$\times$ 
speedup over $N=2$ with peak memory reduced by 8\%, at the cost of RRE 
increasing to 0.1872. Visual results (Figure~\ref{fig:exp_1a_reconstruction_plot}) 
confirm that $N \leq 4$ maintains good edge preservation. 
Table~\ref{tab:ct_static_results} summarizes all results. The effect of inner 
iteration count on streaming performance is studied in the supplementary 
material (Section~S2), where we show that 10 inner iterations provides a 
good speed-accuracy balance.

\begin{figure}[htbp]
\centering
\includegraphics[width=\textwidth]{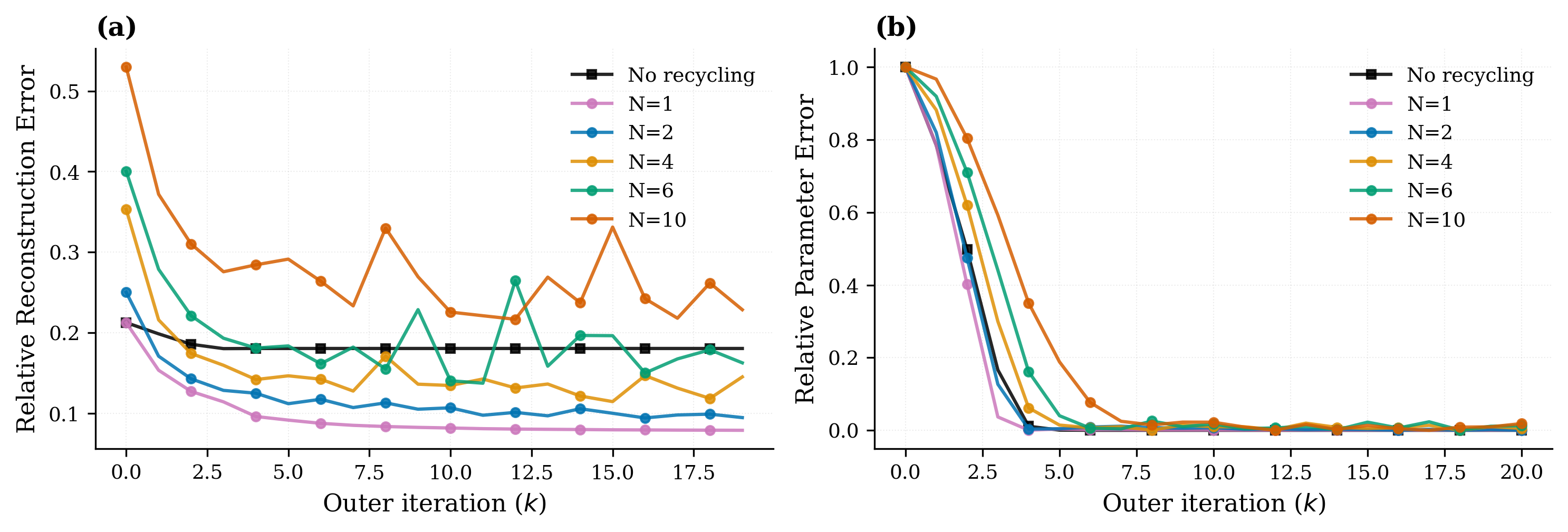}
\caption{Test 1 convergence comparison. (a) RRE vs.\ outer iteration: 
NL-RMM-GKS outperforms standard MM-GKS, with streaming variants trading 
quality for speed. (b) Parameter error converges rapidly for all methods, 
reaching $< 10^{-3}$ by iteration 5.}
\label{fig:exp_1a_convergence_plot}
\end{figure}

\begin{figure}[htbp]
\centering
\includegraphics[width=\textwidth]{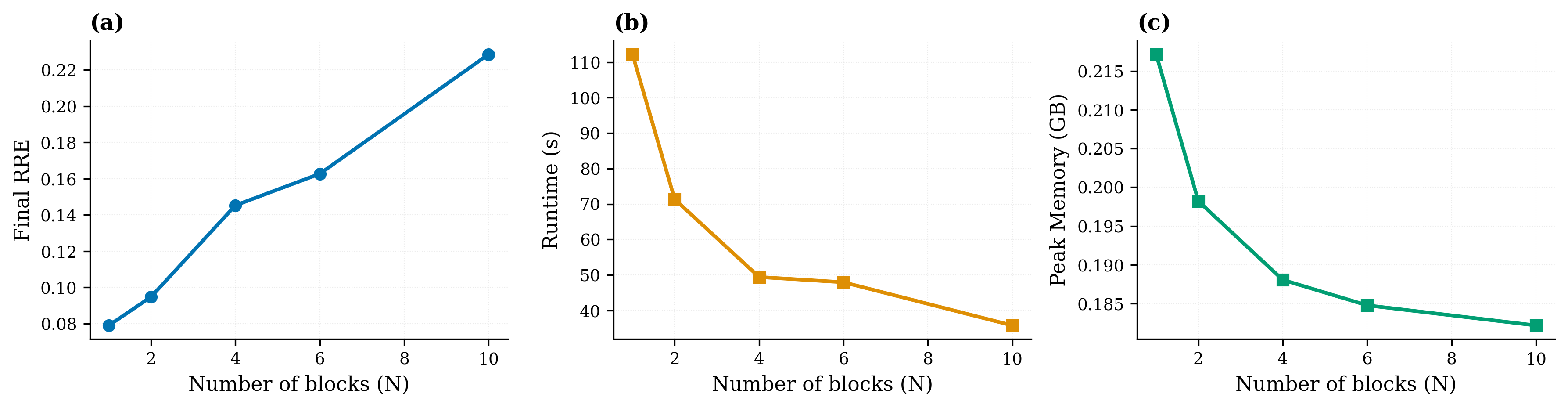}
\caption{Test 1 streaming performance vs.\ number of blocks. 
(a) Final RRE increases with $N$. 
(b) Runtime decreases substantially (5.2$\times$ speedup from $N=2$ to $N=10$). 
(c) Peak memory drops 8\% at $N=10$.}
\label{fig:ct_streaming_blocks}
\end{figure}

\begin{figure}[htbp]
\centering
\includegraphics[width=\textwidth]{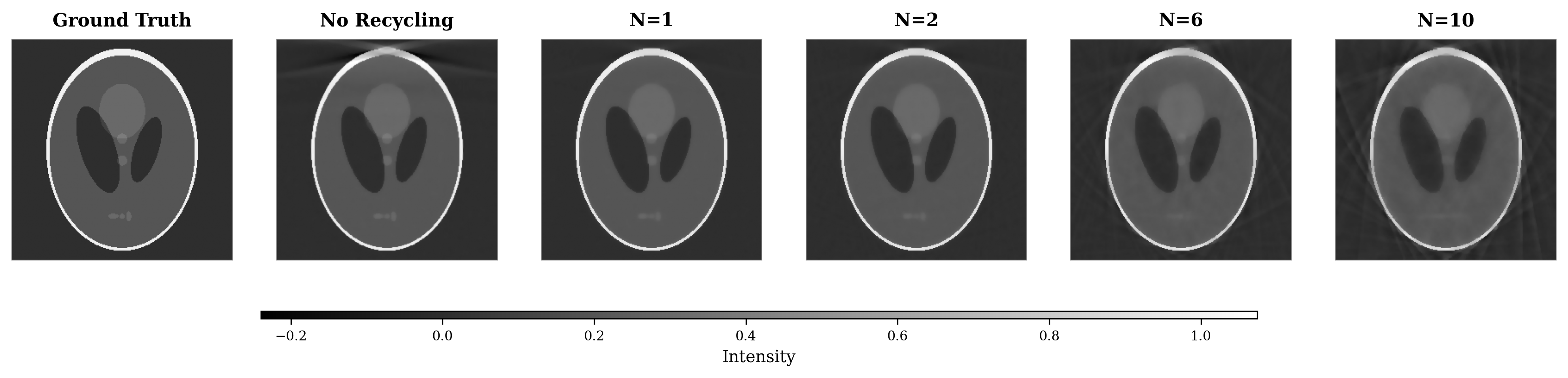}
\caption{Test 1 visual comparison. Shepp-Logan reconstructions for $N = 1, 
2, 6, 10$ blocks compared to standard MM-GKS. Progressive blurring of ellipse 
boundaries is visible as $N$ increases, particularly in small interior 
structures.}
\label{fig:exp_1a_reconstruction_plot}
\end{figure}

\begin{table}[htbp]
\centering
\small
\caption{Test 1: Static Shepp-Logan CT with angle uncertainty 
($p_\text{true} = 0.2421^\circ$, $p^{(0)} = 0^\circ$).}
\label{tab:ct_static_results}
\begin{tabular}{lccccc}
\toprule
Method & Time & Peak Mem. & Final & Param. & $p^*$ \\
       & (min) & (GB)      & RRE   & Err.   & ($^\circ$)   \\
\midrule
MM-GKS (no recycling)    & 2.08 & 0.220 & 0.1805 & 1.49e-4 & 0.2421 \\
NL-RMM-GKS ($N=1$)       & 1.87 & 0.217 & 0.0790 & 7.36e-4 & 0.2419 \\
s-NL-RMM-GKS ($N=2$)     & 1.19 & 0.198 & 0.0948 & 2.89e-4 & 0.2422 \\
s-NL-RMM-GKS ($N=4$)     & 0.82 & 0.188 & 0.1453 & 5.35e-3 & 0.2408 \\
s-NL-RMM-GKS ($N=6$)     & 0.80 & 0.185 & 0.1627 & 1.28e-2 & 0.2390 \\
s-NL-RMM-GKS ($N=10$)    & 0.40 & 0.182 & 0.1872 & 1.87e-2 & 0.2376 \\
\bottomrule
\end{tabular}
\end{table}

\subsection{Photoacoustic Tomography with Sensor Radius Uncertainty.}

In our circular array PAT setup, $n_c$ concentric circles of sensors 
surround the imaging region with nominal radii $r_j^{\text{nom}} = 2j/n_c$. 
A global radial shift $p$ affects all circles identically: $r_j(p) = 
r_j^{\text{nom}} + p$. This uniform perturbation arises physically from 
sound speed variations in the coupling medium, sensor positioning errors, 
or systematic acquisition delays, and motivates joint estimation of $p$ 
alongside the absorption image.

\subsubsection{Test 2: Static Tectonic Phantom---Initialization Sensitivity}

\paragraph{Setup.}
We use a $64 \times 64$ tectonic plate phantom featuring sharp boundaries between regions of different absorption. This phantom is specifically designed to test edge preservation under geometric uncertainty. We introduce a radius perturbation $p_{\text{true}} = 0.432$ units applied uniformly to all circles. 

To assess robustness to initialization quality, we test four scenarios with increasing initial error:
\begin{itemize}
\item \textbf{Very good}: Initial offset $= 0.1$ units from truth
\item \textbf{Good}: Initial offset $= 0.2$ units from truth
\item \textbf{Poor}: Initial offset $= 0.5$ units from truth
\item \textbf{Very poor}: Initial offset $= 0.6$ units from truth
\end{itemize}
We use streaming with $N=3$ blocks and standard recycling parameters ($k_{\min} = 5$, $k_{\max} = 25$).

\paragraph{Goal.}
Investigate the robustness of AltMin vs. VarPro implementations to poor parameter initialization.

\paragraph{Results.}
Figure~\ref{fig:exp_3_error_comparison} reveals a clear pattern: when initialization is good (offsets 0.1, 0.2), both VarPro (solid lines) and AltMin (dashed lines) converge smoothly and achieve similar final reconstruction quality. VarPro actually converges slightly faster in these cases, reaching low RRE by iteration 10.

However, the story changes for poor initialization. At offset $= 0.5$ (green curves), we start to see VarPro exhibit more oscillatory behavior in panel (a), while AltMin maintains steadier convergence. At the worst initialization (offset $= 0.6$), VarPro struggles significantly---its reconstruction error oscillates wildly between 2 and 7 before slowly improving, and it never fully stabilizes even by iteration 20. In contrast, AltMin with offset $= 0.6$ shows some initial oscillations but settles into monotonic convergence after a few iterations.

This difference in robustness makes sense from an algorithmic perspective: VarPro eliminates the image variable $\mathbf{u}$ analytically, creating a reduced but highly nonlinear problem in $p$ alone. When far from the solution, this nonlinearity can cause the optimization landscape to be rough with many local irregularities. AltMin, by alternating between $\mathbf{u}$ and $p$ updates, effectively smooths out this landscape at the cost of requiring more outer iterations.

Panel (b) shows that parameter estimation follows similar trends, with VarPro exhibiting larger oscillations for poor initialization. The visual reconstructions in Figure~\ref{fig:exp_3_reconstruction_comparison} confirm that both methods eventually produce comparable image quality, but AltMin gets there more reliably when starting far from the truth.

We recommend using VarPro when you have a decent initial guess (e.g., from calibration or grid search), but preferring AltMin when initialization quality is uncertain or the problem is particularly challenging.

\begin{figure}[htbp]
\centering
\includegraphics[width=\textwidth]{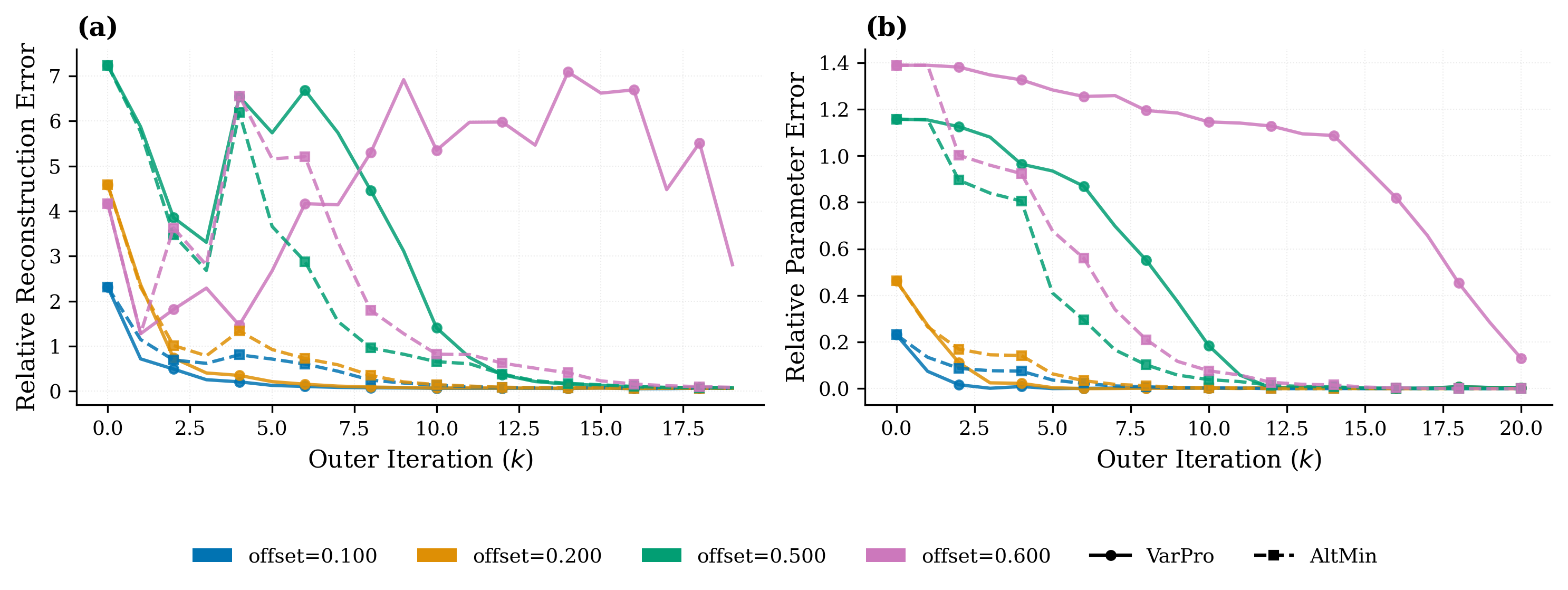}
\caption{Test 2. initialization sensitivity comparison. 
(a) RRE vs. iteration for both VarPro (solid) and AltMin (dashed) across four initialization qualities. Poor initialization (offsets 0.5, 0.6) causes VarPro to oscillate significantly, while AltMin maintains more stable convergence.
(b) Parameter error shows similar trends, with VarPro more sensitive to initialization quality.}
\label{fig:exp_3_error_comparison}
\end{figure}

\begin{figure}[htbp]
\centering
\includegraphics[width=\textwidth]{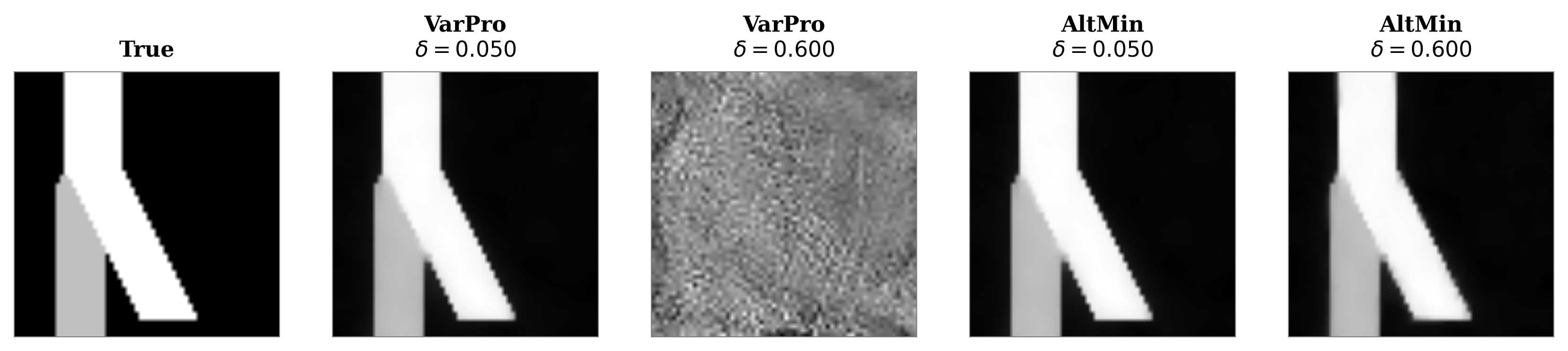}
\caption{Test 2. visual comparison of final reconstructions for different initialization offsets. 
Columns show ground truth, VarPro, and AltMin results across $\delta \in \{0.05, 0.30\}$. Despite convergence differences, both methods achieve similar final visual quality, with AltMin slightly better at preserving sharp boundaries for poor initialization.}
\label{fig:exp_3_reconstruction_comparison}
\end{figure}

\section{Extension of NL-RMM-GKS to Dynamic Problems.}
\label{sec:dynamic}

We now turn to the dynamic setting, where the imaging target changes across
$n_t$ time frames and the forward operator may depend on unknown geometric
parameters $\mathbf{p}$.  The goal is to jointly reconstruct the full
image sequence $\mathbf{u} = \mathrm{vec}([\mathbf{u}^{(1)},\ldots,
\mathbf{u}^{(n_t)}]) \in \mathbb{R}^n$, with $n = n_s n_t$, and to
estimate $\mathbf{p}$, from noisy measurements $\mathbf{b}$.
This section is self-contained: we introduce the dynamic forward model, 
present the full mathematical descriptions of both temporal regularization 
strategies (ANISO-TV and optical flow), give dedicated algorithm descriptions, 
and report numerical experiments on two dynamic phantoms.

\subsection{Dynamic Forward Model and Problem Formulation}
\label{sec:dyn_setup}

In a dynamic inverse problem the forward operator is block-diagonal,
\begin{equation}
  \mathbf{H}(\mathbf{p})
  = \mathrm{blockdiag}\!\left(
      \mathbf{H}_1(\mathbf{p}),\ldots,\mathbf{H}_{n_t}(\mathbf{p})
    \right) \in \mathbb{R}^{m \times n},
\end{equation}
where $\mathbf{H}_t(\mathbf{p}) \in \mathbb{R}^{m_t \times n_s}$ is
the forward operator at time $t$ and $m = \sum_t m_t$.  When the
geometry is the same at every frame, $\mathbf{H}_t(\mathbf{p}) =
\mathbf{H}(\mathbf{p})$ for all $t$ and the block-diagonal structure
reduces to $\mathbf{I}_{n_t} \otimes \mathbf{H}(\mathbf{p})$.  The
joint reconstruction problem takes the form
\begin{equation}
  \label{eq:dyn_joint}
  \min_{\mathbf{u},\mathbf{p}}
  \frac{1}{2}\|\mathbf{H}(\mathbf{p})\mathbf{u} - \mathbf{b}\|_2^2
  + \lambda\|\boldsymbol{\Theta}(\mathbf{u})\,\mathbf{u}\|_1,
\end{equation}
where $\boldsymbol{\Theta}(\mathbf{u})$ is a combined spatio-temporal
regularization operator whose form depends on the chosen temporal
strategy, as described in Sections~\ref{sec:dyn_aniso}
and~\ref{sec:dyn_of} below.

\subsection{Temporal Regularization Strategies}
\label{sec:dyn_reg}

\subsubsection{Anisotropic Space-Time Total Variation (ANISO-TV).}
\label{sec:dyn_aniso}

Following~\cite{pasha2023recycling}, we promote sparse spatial gradients
within each frame and penalize frame-to-frame intensity differences
simultaneously.  Let $\mathbf{L}_s = \begin{bmatrix}
  \mathbf{I}_{n_h} \otimes \mathbf{L}_v \\
  \mathbf{L}_h \otimes \mathbf{I}_{n_v}
\end{bmatrix}$ be the spatial finite-difference operator, where
$\mathbf{L}_v$ and $\mathbf{L}_h$ discretize the first derivative in
the vertical and horizontal directions respectively.  The
ANISO-TV regularization operator is
\begin{equation}
  \label{eq:D1}
  \mathbf{D}_1
  = \begin{bmatrix}
      \mathbf{I}_{n_t} \otimes \mathbf{L}_s \\
      \mathbf{L}_t \otimes \mathbf{I}_{n_s}
    \end{bmatrix},
\end{equation}
where $\mathbf{L}_t$ is the temporal finite-difference operator.  The
regularization term $\|\mathbf{D}_1\mathbf{u}\|_1$ penalizes both
spatial gradients within each frame and differences between consecutive
frames.  Setting $\boldsymbol{\Theta} = \mathbf{D}_1$ (constant,
independent of $\mathbf{u}$) in~\eqref{eq:dyn_joint} recovers the
AnisoTV formulation of~\cite{pasha2023recycling}.  ANISO-TV is well
suited to slowly varying sequences but provides no motion model and
degrades when inter-frame displacements are large.

\subsubsection{Optical Flow Regularization (OF).}
\label{sec:dyn_of}

When the dynamic sequence exhibits coherent motion, we incorporate a
physics-based temporal model following~\cite{okunola2025mmgks-of}.
The optical flow constraint (OFC) assumes that pixel intensities are
preserved as objects move: for a pixel at location $(x_i, y_i)$ with
velocity $\mathbf{s}^i(t) = (s^i_x(t), s^i_y(t))$,
\begin{equation}
  u^i_x(t)\,s^i_x(t) + u^i_y(t)\,s^i_y(t) + u^i_t(t) = 0,
\end{equation}
where $u^i_x, u^i_y, u^i_t$ denote partial derivatives of the image
with respect to $x$, $y$, and $t$.  Assembling this across all pixels
and time pairs gives the system $\boldsymbol{\Upsilon}(\mathbf{u})\,
\mathbf{s} + \mathbf{u}_t = \mathbf{0}$, where
$\boldsymbol{\Upsilon}(\mathbf{u}) = \mathrm{diag}(
\boldsymbol{\Upsilon}(\mathbf{u}^{(1)}), \ldots,
\boldsymbol{\Upsilon}(\mathbf{u}^{(n_t-1)}))$ is a block-diagonal
matrix of spatial image gradients, and $\mathbf{u}_t$ encodes temporal
differences.

\paragraph{Velocity estimation.}
Given the current image estimate $\mathbf{u}^{(k)}$, the forward
velocity field $\mathbf{s}^{(k)}$ is estimated by solving the
regularized least-squares problem
\begin{equation}
  \label{eq:of_forward}
  \mathbf{s}^{(k)}
  = \argmin_{\mathbf{s}}
  \left\|\boldsymbol{\Upsilon}(\mathbf{u}^{(k)})\mathbf{s}
        + \mathbf{u}_t^{(k)}\right\|_p^p
  + \gamma\|\hat{\mathbf{L}}\mathbf{s}\|_q^q,
\end{equation}
where $\hat{\mathbf{L}}$ is a discrete gradient operator that
regularizes the velocity field, and $p, q \in \{1, 2\}$.  The reverse
velocity field $\mathbf{s}'^{(k)}$ is approximated via the relation
$(s'_x + s_x(t), s'_y + s_y(t)) \approx -(s_x(t), s_y(t))$ to
halve computation~\cite{okunola2025mmgks-of}.  Each per-frame subproblem
is solved efficiently with MMGKS using automatic regularization
parameter selection via GCV.

\paragraph{Motion operator encoding.}
Given $\mathbf{s}^{(k)}$ and $\mathbf{s}'^{(k)}$, we construct the
block-structured motion matrices
\begin{equation}
  \bar{\mathbf{M}}(\mathbf{s}^{(k)})
  = \begin{bmatrix}
      \mathbf{I} & -\mathbf{M}(\mathbf{s}^{(k)}(1)) & & \\
      & \ddots & \ddots & \\
      & & \mathbf{I} & -\mathbf{M}(\mathbf{s}^{(k)}(n_t-1))
    \end{bmatrix},
\end{equation}
and $\bar{\mathbf{M}}'(\mathbf{s}'^{(k)})$ analogously for the reverse
flow, where each $\mathbf{M}(\mathbf{s}^{(k)}(t)) \in \mathbb{R}^{n_s
\times n_s}$ is a bilinear warping operator encoding per-frame
displacements~\cite{okunola2025mmgks-of}.  The combined motion
regularization operator is
\begin{equation}
  \hat{\mathbf{M}}(\mathbf{s}^{(k)}, \mathbf{s}'^{(k)})
  = \begin{bmatrix}
      \bar{\mathbf{M}}(\mathbf{s}^{(k)}) \\
      \bar{\mathbf{M}}'(\mathbf{s}'^{(k)})^\top
    \end{bmatrix}.
\end{equation}

\paragraph{Image reconstruction step.}
With fixed velocity estimates, we set the combined regularization
operator
\begin{equation}
  \label{eq:Theta_OF}
  \boldsymbol{\Theta}^{(k)}
  = \begin{bmatrix}
      \boldsymbol{\Psi} \\
      \hat{\mathbf{M}}(\mathbf{s}^{(k)}, \mathbf{s}'^{(k)})
    \end{bmatrix},
\end{equation}
where $\boldsymbol{\Psi}$ is the spatial regularization operator.
Problem~\eqref{eq:dyn_joint} then becomes
\begin{equation}
  \mathbf{u}^{(k+1)}
  = \argmin_{\mathbf{u}}
  \frac{1}{2}\|\mathbf{H}(\mathbf{p}^{(k)})\mathbf{u} - \mathbf{b}\|_2^2
  + \lambda\|\boldsymbol{\Theta}^{(k)}\mathbf{u}\|_1,
\end{equation}
which is solved via RMM-GKS with the discrepancy principle.
Unlike ANISO-TV, $\boldsymbol{\Theta}^{(k)}$ is updated at each outer
iteration as the image estimate improves, so the regularizer adapts to
the evolving motion field.

\subsection{Dynamic NL-RMM-GKS Algorithm}
\label{sec:dyn_alg}

Both ANISO-TV and OF plug directly into the NL-RMM-GKS framework
(Algorithm~\ref{alg:nl_rmmgks_two_solves}) by supplying the appropriate
operator $\boldsymbol{\Theta}^{(k)}$ at each outer iteration.  The key
difference from the static case is the additional velocity estimation
step required by OF before each image solve.  Algorithm~\ref{alg:dyn_nl_rmmgks}
states the unified dynamic procedure; setting $\boldsymbol{\Theta}^{(k)}
= \mathbf{D}_1$ (constant) recovers the ANISO-TV variant, while
executing the \textsc{Solve-OF} subroutine and forming
$\boldsymbol{\Theta}^{(k)}$ via~\eqref{eq:Theta_OF} gives the OF
variant.

\begin{algorithm}
\caption{Dynamic NL-RMM-GKS (with plug-in temporal regularization)}
\label{alg:dyn_nl_rmmgks}
\begin{algorithmic}[1]
\REQUIRE $\mathbf{b}, \boldsymbol{\Psi}, \mathbf{p}^{(0)},
         k_{\min}, k_{\max}, \epsilon,
         \lambda_{\mathrm{hi}}, \tau$
         \hfill // $\tau$: OF update frequency (set $\tau=1$ for every iteration)
\ENSURE $(\mathbf{u}^*, \mathbf{p}^*)$
\STATE Initialize $\mathbf{u}^{(0)}$, $\mathbf{V}^{(0)}_{k_{\min}}$,
       $\boldsymbol{\Theta}^{(0)} \leftarrow \mathbf{D}_1$
       \hfill // same initialization for both strategies
\FOR{$k = 0, 1, 2, \ldots$}
  \STATE $\mathbf{H}_k \leftarrow \mathbf{H}(\mathbf{p}^{(k)})$
  \IF{OF variant \AND $k \bmod \tau = 0$}
    \STATE \textbf{Velocity estimation:} for $t = 1,\ldots,n_t-1$:
    \[
      \mathbf{s}^{(k)}(t)
      = \mathrm{SOLVE\text{-}OF}\!\left(
          \mathbf{u}^{(k)}(t),\, \mathbf{u}^{(k)}(t{+}1),\,
          \hat{\mathbf{L}}\right)
    \]
    \STATE Approximate $\mathbf{s}'^{(k)}$ from $\mathbf{s}^{(k)}$
           via OFC relation (Remark~3.1 of~\cite{okunola2025mmgks-of})
    \STATE $\boldsymbol{\Theta}^{(k)} \leftarrow
           \begin{bmatrix}
             \boldsymbol{\Psi} \\
             \hat{\mathbf{M}}(\mathbf{s}^{(k)}, \mathbf{s}'^{(k)})
           \end{bmatrix}$
  \ELSEIF{ANISO-TV variant}
    \STATE $\boldsymbol{\Theta}^{(k)} \leftarrow \mathbf{D}_1$
           \hfill // fixed; no motion estimation needed
  \ENDIF
  \STATE \textbf{Primary image solve:}
  \[
    (\mathbf{u}^{(k+1)}, \lambda^{(k)}, \mathbf{V}^{(k+1)}_{k_{\min}})
    = \mathrm{RMM\text{-}GKS}(
        \mathbf{H}_k, \boldsymbol{\Theta}^{(k)}, \mathbf{b},
        \mathbf{u}^{(k)}, \mathbf{V}^{(k)}_{k_{\min}},
        k_{\min}, k_{\max}, \epsilon)
  \]
  \STATE \textbf{Optional high-regularization solve:}
  \[
    (\hat{\mathbf{u}}^{(k+1)}, \hat{\lambda}^{(k)}, \hat{\mathbf{V}}_{k_{\min}})
    = \mathrm{RMM\text{-}GKS}(
        \mathbf{H}_k, \boldsymbol{\Theta}^{(k)}, \mathbf{b},
        \mathbf{u}^{(k+1)}, \mathbf{V}^{(k+1)}_{k_{\min}},
        k_{\min}, k_{\max}, \epsilon;\,
        \lambda_{\mathrm{hi}})
  \]
  \hfill // $\hat{\mathbf{V}}_{k_{\min}}$ discarded after this step
  \STATE \textbf{Parameter update:}
  \[
    \mathbf{p}^{(k+1)}
    = \mathrm{UPDATE\text{-}PARAM}(
        \mathbf{H}_k, \boldsymbol{\Theta}^{(k)},
        \hat{\mathbf{u}}^{(k+1)}, \mathbf{b},
        \mathbf{p}^{(k)}, \hat{\lambda}^{(k)})
  \]
  \STATE \textbf{Update weights} $\mathbf{P}^{(k+1)}_\epsilon$
         from $\mathbf{u}^{(k+1)}$
  \STATE \textbf{Compute and project residual:}
  \STATE $\mathbf{r}^{(k+1)} \leftarrow
         \mathbf{H}_k^\top(\mathbf{H}_k \mathbf{u}^{(k+1)} - \mathbf{b})
         + \lambda^{(k)}
           (\boldsymbol{\Theta}^{(k)})^\top
           (\mathbf{P}^{(k+1)}_\epsilon)^2
           \boldsymbol{\Theta}^{(k)} \mathbf{u}^{(k+1)}$
  \STATE $\mathbf{r}^{(k+1)} \leftarrow \mathbf{r}^{(k+1)}
         - \mathbf{V}^{(k+1)}_{k_{\min}}
           (\mathbf{V}^{(k+1)}_{k_{\min}})^\top \mathbf{r}^{(k+1)}$
  \STATE $\mathbf{V}^{(k+1)}_{k_{\min}} \leftarrow
         [\mathbf{V}^{(k+1)}_{k_{\min}},\;
          \mathbf{r}^{(k+1)} / \|\mathbf{r}^{(k+1)}\|_2]$
  \STATE Check outer convergence
\ENDFOR
\STATE $\mathbf{u}^* = \mathbf{u}^{(k+1)},\quad
       \mathbf{p}^* = \mathbf{p}^{(k+1)}$
\end{algorithmic}
\end{algorithm}

\paragraph{Streaming extension.}
The streaming variant of the dynamic algorithm (s-NL-RMM-GKS, dynamic)
follows Algorithm~\ref{alg:s_nl_rmmgks_final} verbatim, with
$\boldsymbol{\Theta}^{(k)}$ replaced by the dynamic operator at each
block.  Velocity estimation for OF is performed once per outer block
using the current image estimate before the primary image solve; the
estimated velocity field is then held fixed within that block.  When
$\mathbf{p}$ is known and fixed, and ANISO-TV is used, this reduces
exactly to s-RMM-GKS applied to a dynamic linear problem~\cite{pasha2023recycling}.

\paragraph{Relationship to prior work.}
When $\mathbf{p}$ is known, Algorithm~\ref{alg:dyn_nl_rmmgks} with
ANISO-TV reduces to the dynamic MM-GKS framework of~\cite{pasha2023recycling},
and with OF reduces to MMGKS-OF~\cite{okunola2025mmgks-of}.
NL-RMM-GKS therefore strictly generalizes both methods to the case of
uncertain forward operators.

\subsection{Experimental Setup}
\label{sec:dyn_exp_setup}

All quality metrics, stopping criteria, and common algorithmic parameters
are as described in Section~\ref{sec:experiments}.  

\subsection{Test 3: Dynamic Moving MNIST Sequence (CT)}
\label{sec:dyn_test_mnist}

\paragraph{Setup.}
We use the Moving MNIST dataset~\cite{moving_mnist}, consisting of
$n_t = 20$ frames of size $64 \times 64$ pixels showing two smoothly
translating digits.  For each frame we acquire $n_\theta = 10$ randomly
assigned fan-beam CT projections from $[0^\circ, 180^\circ)$ with
$n_r = 91$ detector bins, representing a limited-angle sparse
acquisition.  A common angular perturbation $p_{\mathrm{true}} =
-0.1480^\circ$ is applied to all frames, with initialization
$p^{(0)} = -0.25^\circ$.  For streaming experiments, the 10 angles per
frame are partitioned into $N$ randomly assigned blocks of approximately
equal size.

\paragraph{Goal.}
We compare the two temporal regularization strategies---optical flow
(OF)~\cite{okunola2025mmgks-of} and anisotropic total variation
(ANISO-TV)~\cite{pasha2023recycling}---within the NL-RMM-GKS framework,
and investigate how streaming block count $N$ affects each strategy.

\paragraph{Results.}
Figure~\ref{fig:exp_2_error_comparison} reveals a striking difference
between strategies.  Optical flow consistently achieves lower
reconstruction error across all $N$: with $N=1$, OF converges to
$\mathrm{RRE} \approx 0.07$ while ANISO-TV reaches only $\approx 0.30$.
This gap widens with $N$; at $N=10$, OF maintains $\mathrm{RRE} =
0.41$ while ANISO-TV stalls at $\mathrm{RRE} = 0.67$.  The superiority
of optical flow is expected here: the Moving MNIST digits undergo smooth
translation, exactly the motion structure that OF is designed to capture,
whereas ANISO-TV enforces temporal smoothness without any motion model
and is better suited to sequences with small inter-frame changes.
Parameter estimation (Figure~\ref{fig:exp_2_error_comparison}(b)) is
more robust than reconstruction to the choice of regularization, though
ANISO-TV becomes unstable at large $N$.  Visual results
(Figure~\ref{fig:exp_2_reconstruction_comparison}) confirm that OF
preserves digit shapes clearly even at $N=10$, while ANISO-TV produces
increasingly unrecognizable digits.  Full results are summarized in
Table~\ref{tab:moving_mnist_results}.

\begin{table}[htbp]
\centering
\caption{Test 3. Dynamic CT for Moving MNIST ($n_t=20$ frames of $64\times64$, 
10 angles per frame, $p_\text{true} = -0.1480^\circ$, $p^{(0)} = -0.25^\circ$). Optical 
flow consistently outperforms ANISO-TV, especially for larger $N$.}
\label{tab:moving_mnist_results}
\begin{tabular}{lccc}
\toprule
Method & Final RRE & Param.\ Err. & $p^*$ ($^\circ$) \\
\midrule
\multicolumn{4}{l}{\textit{Optical Flow Regularization}} \\
NL-RMM-GKS-OF ($N=1$)       & 0.0719 & 0.0016 & -0.1477 \\
s-NL-RMM-GKS-OF ($N=2$)     & 0.0979 & 0.0051 & -0.1487 \\
s-NL-RMM-GKS-OF ($N=5$)     & 0.2089 & 0.0538 & -0.1400 \\
s-NL-RMM-GKS-OF ($N=10$)    & 0.4071 & 0.0121 & -0.1460 \\
\midrule
\multicolumn{4}{l}{\textit{Anisotropic TV Regularization}} \\
NL-RMM-GKS-ANISO ($N=1$)    & 0.3040 & 0.1074 & -0.1321 \\
s-NL-RMM-GKS-ANISO ($N=2$)  & 0.3605 & 0.0838 & -0.1355 \\
s-NL-RMM-GKS-ANISO ($N=5$)  & 0.5030 & 0.1795 & -0.1214 \\
s-NL-RMM-GKS-ANISO ($N=10$) & 0.6712 & 0.4634 & -0.2165 \\
\bottomrule
\end{tabular}
\end{table}

\begin{figure}[htbp]
\centering
\includegraphics[width=\textwidth]{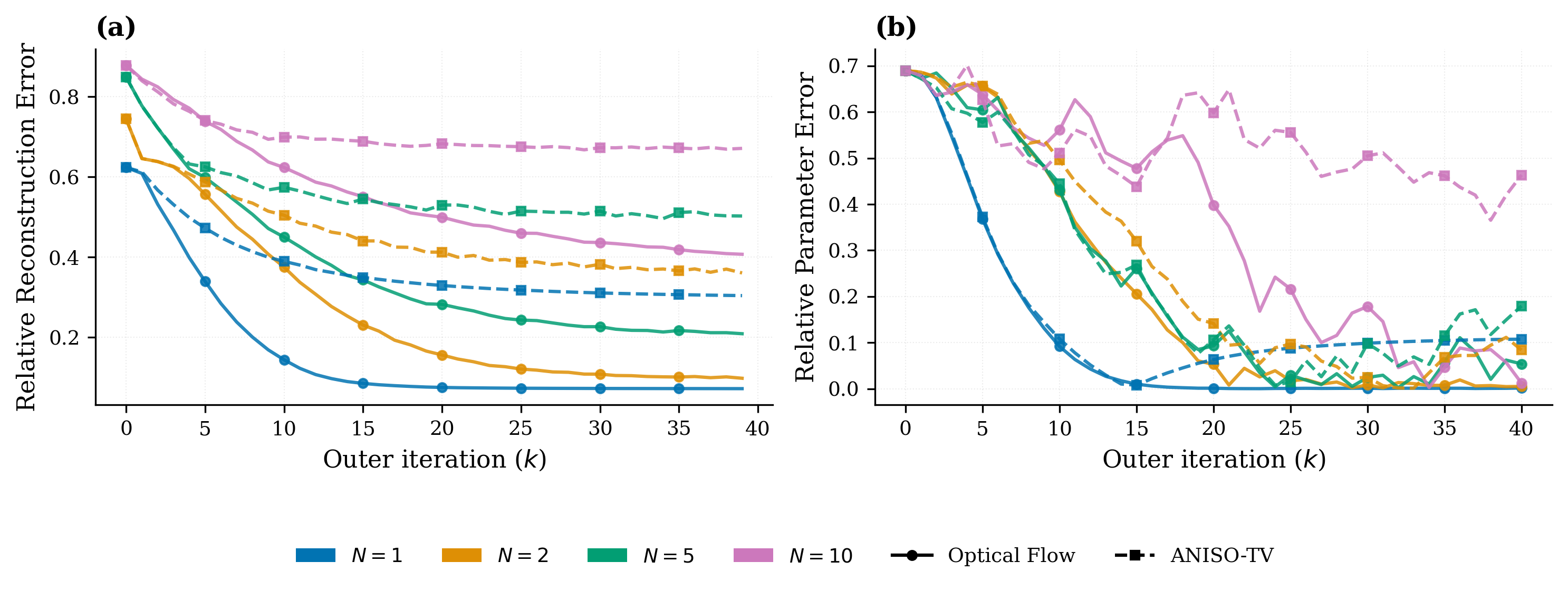}
\caption{Test 3. Convergence comparison across block counts and regularization 
strategies. (a) RRE vs.\ iteration: optical flow (solid) maintains stable 
convergence across all $N$, while ANISO-TV (dashed) degrades significantly 
at $N=10$. (b) Parameter error: both methods converge initially, but ANISO-TV 
becomes unstable at large $N$.}
\label{fig:exp_2_error_comparison}
\end{figure}

\begin{figure}[htbp]
\centering
\includegraphics[width=\textwidth]{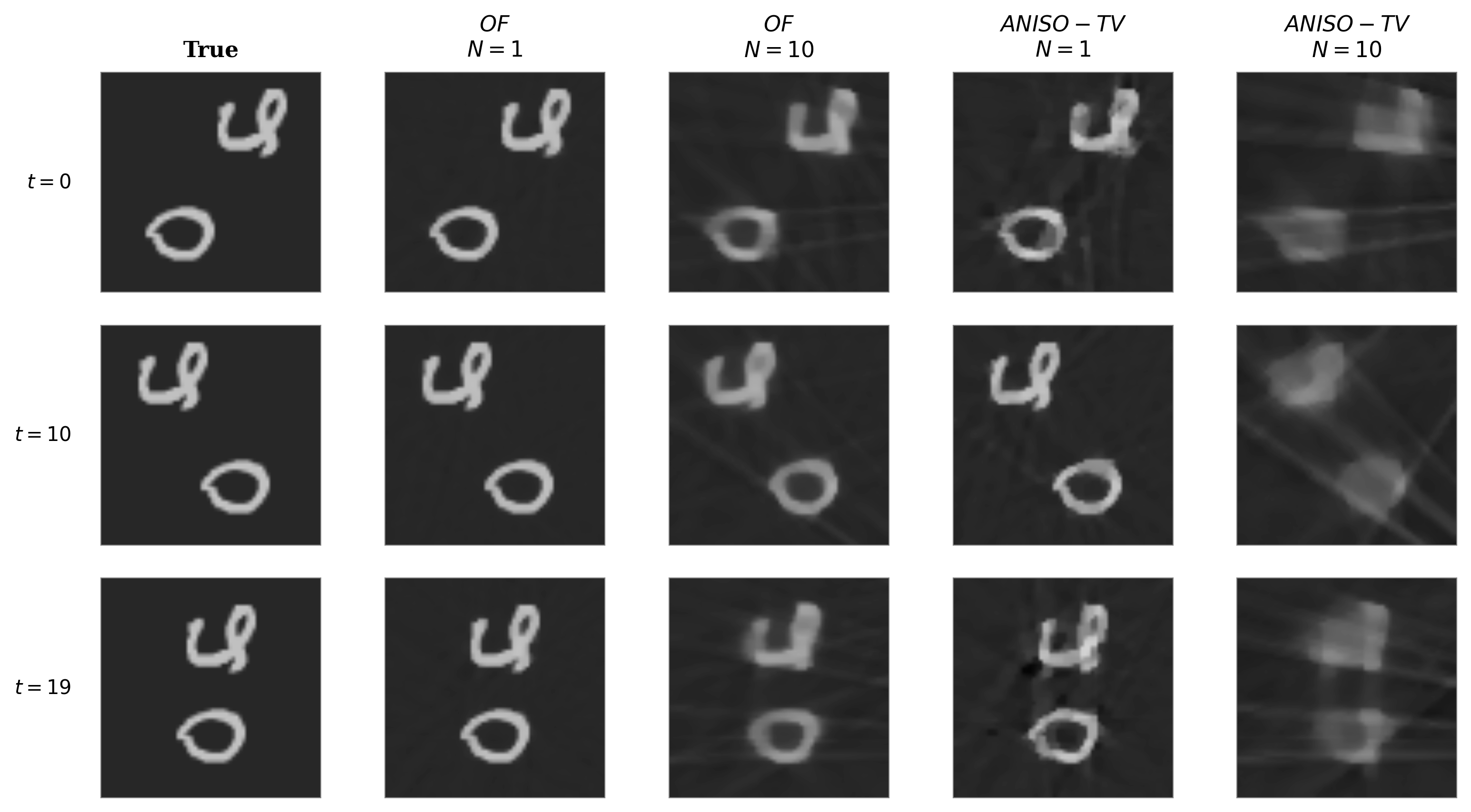}
\caption{Test 3. Visual comparison at three time points ($t=0, 10, 19$) for 
$N \in \{1, 10\}$. Optical flow preserves digit shapes significantly better 
than ANISO-TV, especially at $N=10$ where ANISO-TV fails completely.}
\label{fig:exp_2_reconstruction_comparison}
\end{figure}
\subsection{Test 4: Dynamic Blocks Phantom (PAT)}
\label{sec:dyn_test_pat}

\paragraph{Setup.}
We consider $n_t = 10$ frames of size $50 \times 50$ pixels containing
four rectangular blocks of smoothly varying intensity, imaged with a
circular array PAT sensor.  A global radial shift $p_{\mathrm{true}} =
0.1924$ units is introduced with initialization $p^{(0)} = 0.25$.  We
evaluate performance under recycling window sizes $k_{\max} \in
\{10, 15, 20, 25\}$ with $k_{\min} = 5$ fixed, comparing OF and
ANISO-TV regularization.  A baseline that uses only the coarse
grid-search parameter estimate and then solves the linear problem is
included to demonstrate the benefit of joint estimation.

\paragraph{Goal.}
To determine the optimal recycling window size for dynamic PAT problems,
confirm the superiority of OF regularization in this setting, and
demonstrate the benefit of full joint estimation over the grid-search
baseline.

\paragraph{Results.}
Figure~\ref{fig:exp_4_error_comparison}(a) shows reconstruction error
decreasing steadily for all recycling window sizes, with curves tightly
clustered, indicating that reconstruction quality is relatively
insensitive to the window size within the tested range.  Optical flow
consistently outperforms ANISO-TV by roughly $2\times$ in final RRE
($0.11$--$0.13$ vs.\ $0.26$--$0.28$).  Panel~(b) shows that parameter
estimation is even less sensitive to $k_{\max}$: all configurations
converge to parameter errors below $0.01$ by iteration~20.  As shown
in Table~\ref{tab:test4_results}, increasing $k_{\max}$ from $10$ to
$25$ reduces RRE from $0.1322$ to $0.1113$ for OF---a modest $16\%$
improvement that must be weighed against the proportional memory
increase, suggesting $k_{\max} \in [15, 20]$ as a practical sweet spot.
Visual results (Figure~\ref{fig:exp_4_reconstruction_comparison}) confirm
that OF preserves block edges well across all $k_{\max}$ values, while
ANISO-TV produces noticeably blurred boundaries.  Critically, the
grid-search-only baseline produces substantially worse reconstructions
than joint estimation under either regularization strategy, demonstrating
that iterative parameter refinement is essential.

\begin{figure}[htbp]
\centering
\includegraphics[width=\textwidth]{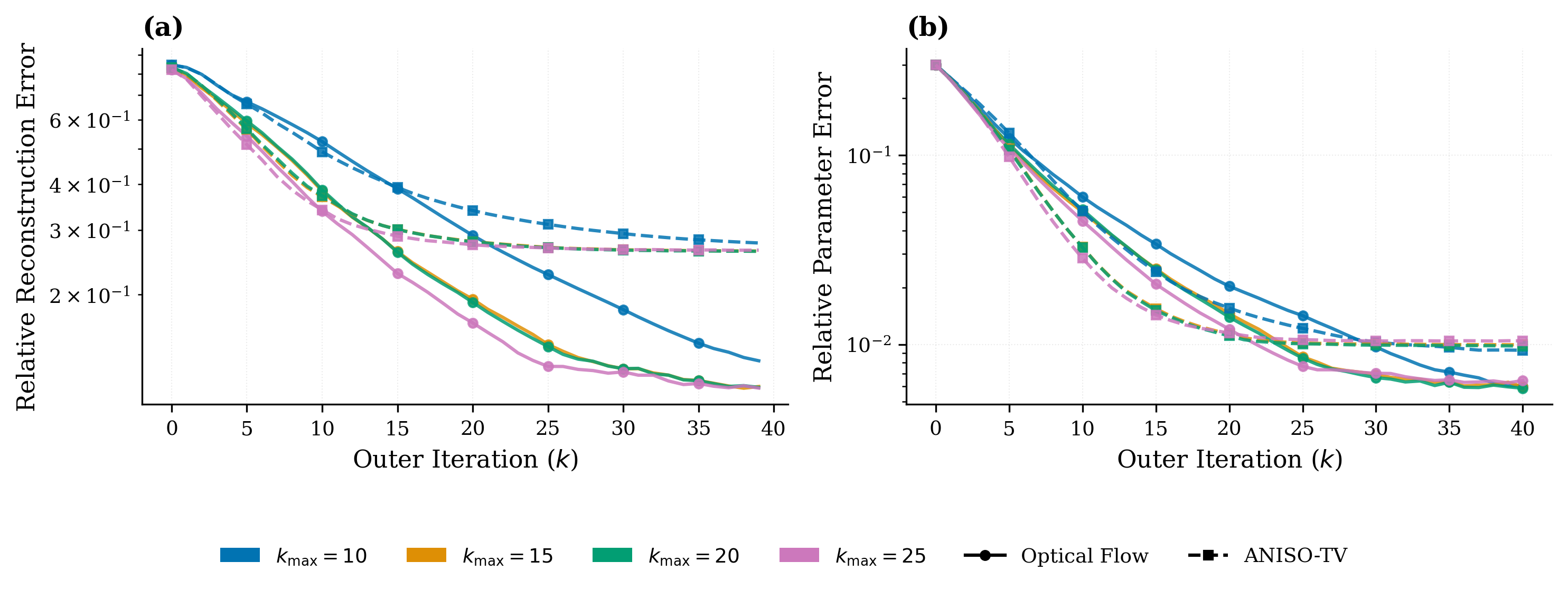}
\caption{Test 4. Effect of recycling window size on dynamic PAT reconstruction. 
(a) RRE vs.\ iteration for optical flow (solid) and ANISO-TV (dashed) across 
$k_{\max} \in \{10, 15, 20, 25\}$. Curves are tightly clustered, showing 
reconstruction is relatively insensitive to window size. (b) Parameter error 
converges rapidly for all settings.}
\label{fig:exp_4_error_comparison}
\end{figure}

\begin{figure}[htbp]
\centering
\includegraphics[width=\textwidth]{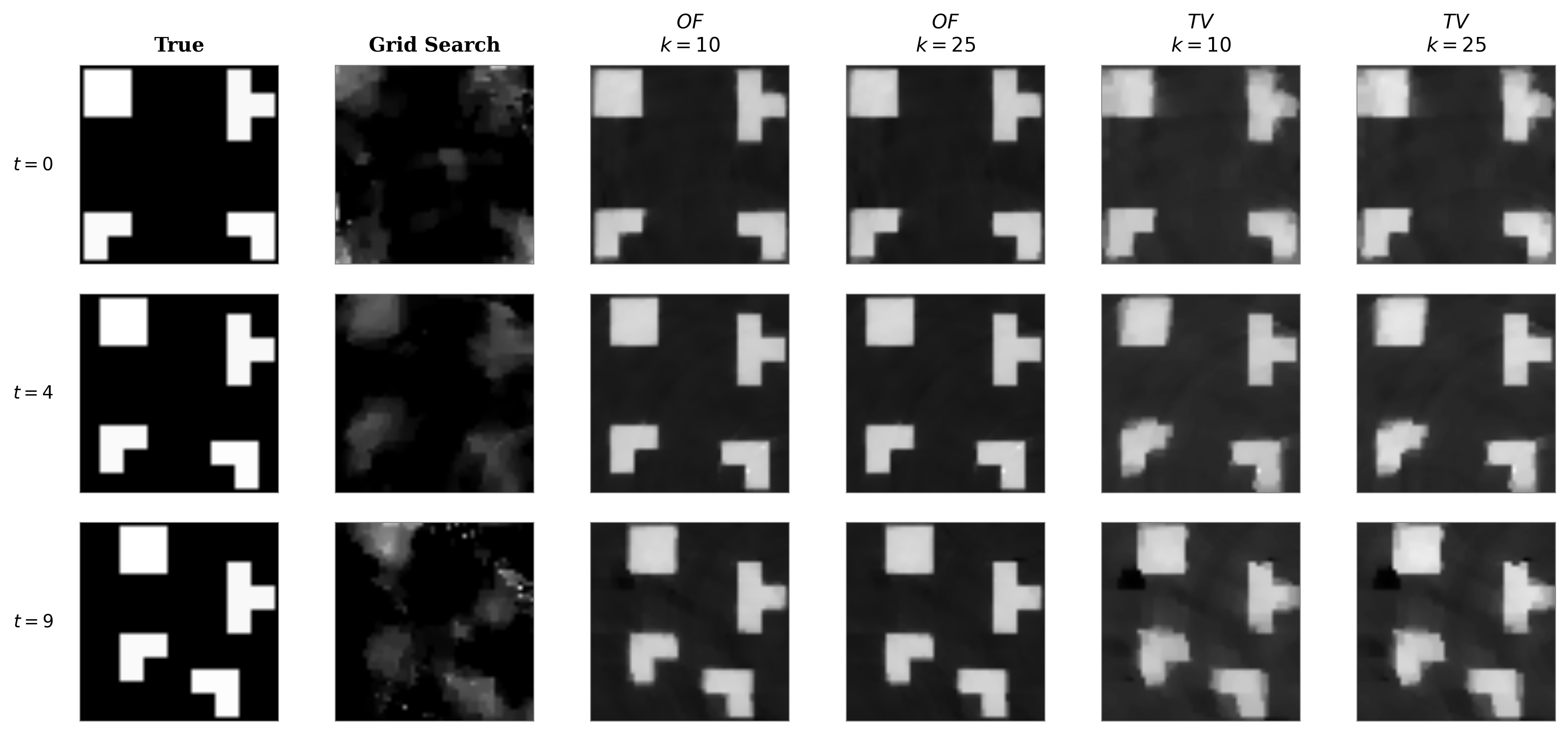}
\caption{Test 4. Visual comparison at three time points ($t=0, 4, 9$) for 
$k_{\max} \in \{10, 25\}$. Columns compare ground truth, grid-search-only 
baseline, optical flow (OF), and ANISO-TV (TV). Optical flow preserves sharp 
block edges significantly better than ANISO-TV; differences between $k_{\max} 
= 10$ and $k_{\max} = 25$ are subtle.}
\label{fig:exp_4_reconstruction_comparison}
\end{figure}

\begin{table}[htbp]
\centering
\caption{Test 4: Dynamic PAT for blocks phantom (10 frames of $50\times50$, 
$p_\text{true} = 0.1924$, $p^{(0)} = 0.25$). Optical flow consistently 
outperforms ANISO-TV. Modest recycling windows ($k_{\max} = 15$) provide a 
good balance between memory and accuracy.}
\label{tab:test4_results}
\begin{tabular}{lccc}
\toprule
Method & Final RRE & Param.\ Err. & $p^*$ \\
\midrule
\multicolumn{4}{l}{\textit{Optical Flow Regularization}} \\
NL-RMM-GKS-OF ($k_{\max}=10$) & 0.1322 & 0.0060 & 0.1935 \\
NL-RMM-GKS-OF ($k_{\max}=15$) & 0.1125 & 0.0060 & 0.1935 \\
NL-RMM-GKS-OF ($k_{\max}=20$) & 0.1118 & 0.0058 & 0.1935 \\
NL-RMM-GKS-OF ($k_{\max}=25$) & 0.1113 & 0.0065 & 0.1936 \\
\midrule
\multicolumn{4}{l}{\textit{Anisotropic TV Regularization}} \\
NL-RMM-GKS-ANISO ($k_{\max}=10$) & 0.2767 & 0.0093 & 0.1942 \\
NL-RMM-GKS-ANISO ($k_{\max}=15$) & 0.2631 & 0.0010 & 0.1943 \\
NL-RMM-GKS-ANISO ($k_{\max}=20$) & 0.2625 & 0.0098 & 0.1943 \\
NL-RMM-GKS-ANISO ($k_{\max}=25$) & 0.2645 & 0.0104 & 0.1944 \\
\bottomrule
\end{tabular}
\end{table}

\section{Conclusions and Outlook.}
\label{sec:conclusions}

We have developed a comprehensive framework for nonlinear inverse problems 
with uncertain forward operators by extending the recycled majorization-minimization 
generalized Krylov subspace method to the nonlinear setting. The framework 
addresses the fundamental challenge that arises when the forward operator 
depends on unknown geometric or calibration parameters, coupling a 
large-scale image reconstruction problem with a nonlinear parameter 
estimation problem.

Our main contributions are:
\begin{enumerate}
\item Two complementary realizations --- AltMin and VarPro --- that couple 
image reconstruction with parameter estimation while maintaining bounded 
memory through Krylov recycling. Both are supported by a convergence 
guarantee to a stationary point of the joint objective 
(Theorem~\ref{thm:convergence}, Appendix~\ref{sec:appendix_proof}).

\item Streaming extensions that process sequential data blocks with basis 
recycling, enabling reconstruction from large-scale datasets without 
ever forming or storing the full forward operator.

\item A unified framework that generalizes MM-GKS, RMM-GKS, s-RMM-GKS, 
and MMGKS-OF as special cases, and naturally accommodates plug-in temporal 
regularization strategies including optical flow and anisotropic total 
variation for dynamic imaging problems.
\end{enumerate}

Numerical experiments in CT and PAT demonstrate that the recycling strategy 
successfully bounds memory growth --- preventing the unbounded subspace 
expansion that afflicts standard MM-GKS --- while simultaneously improving 
reconstruction quality by maintaining a high-quality compressed basis 
throughout iterations. Streaming variants offer a practical mechanism for 
handling large or sequentially acquired datasets, with the number of blocks 
$N$ providing a controllable trade-off between memory, runtime, and 
reconstruction quality. For dynamic problems, optical flow regularization 
substantially outperforms anisotropic total variation when inter-frame 
motion is large and structured, while ANISO-TV remains a simpler and 
effective choice when frame-to-frame changes are small.

Several directions remain open for future investigation:

\paragraph{Convergence rates.} 
Theorem~\ref{thm:convergence} establishes convergence to a stationary 
point but does not characterize the rate of convergence. Establishing 
linear or sublinear convergence rates --- potentially under additional 
assumptions on the objective --- would provide a more complete theoretical foundation and 
guide practical parameter selection.

\paragraph{Extension to multiple and distributed parameters.} 
The current framework estimates a single global geometric parameter $p$ 
per problem. Many practical systems require estimation of multiple 
independent parameters simultaneously --- for example, per-angle 
perturbations in CT or per-sensor position errors in PAT. Extending 
NL-RMM-GKS to high-dimensional parameter spaces will require careful 
treatment of the Gauss-Newton system as $n_p$ grows, potentially 
incorporating structure such as sparsity or low-rank constraints on 
the parameter vector.

\paragraph{Real data validation.}
All experiments in this paper use simulated data with known ground truth. 
Validation on real CT and PAT acquisitions with genuine geometric 
uncertainties is an important next step, requiring careful treatment of 
model mismatch, non-Gaussian noise, and the absence of ground truth for 
quantitative evaluation.

\paragraph{Learned initialization.}
The coarse grid search used for parameter initialization is effective but 
becomes expensive as $n_p$ grows. A natural extension is to replace the 
grid search with a learned predictor --- for example, a small network 
trained to predict geometric parameters from sinogram features --- which 
could provide fast, reliable initialization even for high-dimensional 
parameter spaces, while the subsequent NL-RMM-GKS refinement retains 
the theoretical guarantees of the optimization-based approach.

\paragraph{Adaptive block partitioning for streaming.}
The current streaming implementation uses fixed randomly assigned blocks. 
An adaptive strategy that partitions data based on information content could improve reconstruction quality at a given $N$ 
and reduce the number of passes required for convergence.

\section*{Acknowledgments}
MP acknowledges support from NSF DMS 2410699, MK acknowledges support from NSF DMS 2410698, and JN acknowledges support from NSF DMS-2038118. Any opinions, findings, conclusions or recommendations expressed in this material are those of the authors and do not necessarily reflect the views of the National Science Foundation. 
\bibliographystyle{siamplain}
\bibliography{references}

\appendix
\section{Proof of Theorem~\ref{thm:convergence}.}
\label{sec:appendix_proof}

We present the complete proof of Theorem~\ref{thm:convergence}. 
We first state the assumptions precisely, then establish a sequence 
of lemmas before combining them in the main proof. The extension to 
multiple inner iterations is given in Theorem~\ref{thm:convergence_multi}.

\subsection{Assumptions}

\begin{assumption}[Smoothness of $\mathbf{H}(\mathbf{p})$]
\label{assump:smoothness}
The matrix-valued function $\mathbf{H}(\mathbf{p})$ is twice continuously 
differentiable. There exist constants $C_H, C_{HH} > 0$ such that for all 
$\mathbf{p} \in \mathbb{R}^{n_p}$ and all $i, j \in \{1, \ldots, n_p\}$:
\begin{equation}
\left\|\frac{\partial \mathbf{H}(\mathbf{p})}{\partial p_i}\right\|_2 \leq C_H, 
\qquad
\left\|\frac{\partial^2 \mathbf{H}(\mathbf{p})}{\partial p_i \partial p_j}
\right\|_2 \leq C_{HH}.
\end{equation}
\end{assumption}

\begin{assumption}[Bounded iterates and operators]
\label{assump:bounded}
The iterates remain bounded: there exist $R_u, R_p > 0$ such that 
$\|\mathbf{u}^{(k)}\|_2 \leq R_u$ and $\|\mathbf{p}^{(k)}\|_2 \leq R_p$ 
for all $k$. The operators satisfy $\|\mathbf{H}(\mathbf{p})\|_2 \leq 
C_{\mathbf{H}}$ for all $\|\mathbf{p}\|_2 \leq R_p$, $\|\mathbf{b}\|_2 
\leq C_{\mathbf{b}}$, and $\|\boldsymbol{\Psi}\|_2 \leq C_{\boldsymbol{\Psi}}$ 
for positive constants $C_{\mathbf{H}}, C_{\mathbf{b}}, C_{\boldsymbol{\Psi}}$.
\end{assumption}

\begin{assumption}[Gauss-Newton matrix conditioning]
\label{assump:gn}
The damped Gauss-Newton matrix $\mathbf{J}_{\mathbf{p}}^{(k,\ell)}$ satisfies
\begin{equation}
\gamma_{\min} \mathbf{I} \preceq \mathbf{J}_{\mathbf{p}}^{(k,\ell)} 
\preceq \gamma_{\max} \mathbf{I}
\end{equation}
for all $k, \ell$ and constants $0 < \gamma_{\min} \leq \gamma_{\max} < \infty$, 
where $\gamma_{\min} \geq \mu$ is ensured by the damping parameter $\mu > 0$.
\end{assumption}

These assumptions are standard in nonlinear optimization. 
Assumption~\ref{assump:smoothness} holds for CT with angular perturbations 
and PAT with radial shifts. Assumption~\ref{assump:bounded} can be enforced 
via damping and trust-region strategies. Assumption~\ref{assump:gn} is 
guaranteed by the positive damping parameter $\mu > 0$.

\subsection{Proof Strategy}

The proof establishes that each outer iteration of 
Algorithm~\ref{alg:nl_rmmgks_two_solves} produces sufficient descent in 
$\mathcal{J}_{\epsilon,\lambda}$, and since the objective is bounded below, 
both gradient norms must be summable and hence converge to zero. The argument 
proceeds through six lemmas:

\begin{enumerate}
\item \textbf{Lemma~\ref{lem:Q_bound}:} Bound the spectral norm of the 
majorant Hessian $\mathbf{Q}^{(k)}$ uniformly over all iterations.
\item \textbf{Lemma~\ref{lem:u_descent}:} Show the image update produces 
descent of at least $\|\nabla_{\mathbf{u}}\mathcal{J}\|_2^2 / (2\bar{\mu})$.
\item \textbf{Lemma~\ref{lem:lipschitz_grad}:} Establish Lipschitz continuity 
of $\nabla_{\mathbf{p}}\mathcal{J}$ with explicit constant $L_{\mathbf{p}}$.
\item \textbf{Lemmas~\ref{lem:gn_descent} and~\ref{lem:gn_norm}:} Show the 
Gauss-Newton direction is a descent direction with bounded norm.
\item \textbf{Lemma~\ref{lem:step_size}:} Establish a positive lower bound 
$\alpha_{\min}$ on the backtracking step size.
\item \textbf{Lemma~\ref{lem:p_descent}:} Show the parameter update produces 
descent of at least $c_1\alpha_{\min}\gamma_{\max}^{-1}
\|\nabla_{\mathbf{p}}\mathcal{J}\|_2^2$.
\end{enumerate}

\subsection{Lemmas}

\begin{lemma}[Bound on majorant Hessian]
\label{lem:Q_bound}
Define the quadratic majorant Hessian
\begin{equation}
\mathbf{Q}^{(k)}(\mathbf{p}) := \mathbf{H}(\mathbf{p})^\top \mathbf{H}(\mathbf{p}) 
+ \lambda \boldsymbol{\Psi}^\top (\mathbf{P}^{(k)}_\varepsilon)^2 \boldsymbol{\Psi},
\end{equation}
where $\mathbf{P}^{(k)}_\varepsilon = \mathrm{diag}(((\boldsymbol{\Psi}
\mathbf{u}^{(k)})_j^2 + \varepsilon^2)^{-1/4})$. Then
\begin{equation}
\bar{\mu} := \sup_{k \geq 0} \|\mathbf{Q}^{(k)}(\mathbf{p}^{(k)})\|_2 
\leq C_{\mathbf{H}}^2 + \frac{\lambda}{\varepsilon} C_{\boldsymbol{\Psi}}^2 
< \infty.
\label{eq:mu_bound}
\end{equation}
\end{lemma}

\begin{proof}
By the triangle inequality and submultiplicativity:
\[
\|\mathbf{Q}^{(k)}(\mathbf{p}^{(k)})\|_2 \leq \|\mathbf{H}(\mathbf{p}^{(k)})\|_2^2 
+ \lambda \|\boldsymbol{\Psi}\|_2^2 \|\mathbf{P}^{(k)}_\varepsilon\|_2^2
\leq C_{\mathbf{H}}^2 + \lambda C_{\boldsymbol{\Psi}}^2 \cdot \varepsilon^{-1},
\]
where we used Assumption~\ref{assump:bounded} and the fact that each diagonal 
entry of $\mathbf{P}^{(k)}_\varepsilon$ satisfies 
$((\boldsymbol{\Psi}\mathbf{u}^{(k)})_j^2 + \varepsilon^2)^{-1/2} \leq 
\varepsilon^{-1}$.
\end{proof}

\begin{lemma}[Descent from image update]
\label{lem:u_descent}
Define the quadratic majorant
\begin{equation}
Q(\mathbf{u}, \mathbf{u}^{(k)}; \mathbf{p}^{(k)}) = \frac{1}{2} \mathbf{u}^\top 
\mathbf{Q}^{(k)}(\mathbf{p}^{(k)}) \mathbf{u} - \mathbf{u}^\top 
\mathbf{H}(\mathbf{p}^{(k)})^\top \mathbf{b} + c^{(k)},
\label{eq:Q_def}
\end{equation}
where $c^{(k)}$ is independent of $\mathbf{u}$, and let 
$\mathbf{r}^{(k)} = \nabla_{\mathbf{u}} \mathcal{J}_{\varepsilon,\lambda}
(\mathbf{u}^{(k)}; \mathbf{p}^{(k)})$. If
\begin{equation}
\mathbf{u}^{(k+1)} = \arg\min_{\mathbf{u} \in \mathrm{range}
[\mathbf{V}^{(k)},\, \mathbf{r}^{(k)}]} Q(\mathbf{u}, \mathbf{u}^{(k)}; 
\mathbf{p}^{(k)}),
\end{equation}
then
\begin{equation}
\mathcal{J}_{\varepsilon,\lambda}(\mathbf{u}^{(k+1)}; \mathbf{p}^{(k)}) 
\leq \mathcal{J}_{\varepsilon,\lambda}(\mathbf{u}^{(k)}; \mathbf{p}^{(k)}) 
- \frac{\|\mathbf{r}^{(k)}\|_2^2}{2\bar{\mu}}.
\label{eq:descent_bound}
\end{equation}
\end{lemma}

\begin{proof}
Since $\mathbf{r}^{(k)} \in \mathrm{range}[\mathbf{V}^{(k)}, \mathbf{r}^{(k)}]$, 
the minimizer $\mathbf{u}^{(k+1)}$ is at least as good as the exact line search 
along $-\mathbf{r}^{(k)}$. The exact line search step size is
\[
\alpha^* = \frac{\|\mathbf{r}^{(k)}\|_2^2}{(\mathbf{r}^{(k)})^\top 
\mathbf{Q}^{(k)}(\mathbf{p}^{(k)}) \mathbf{r}^{(k)}},
\]
and substituting $\mathbf{u}^{(k)} - \alpha^* \mathbf{r}^{(k)}$ into 
\eqref{eq:Q_def} gives
\[
Q(\mathbf{u}^{(k)} - \alpha^* \mathbf{r}^{(k)}, \mathbf{u}^{(k)}; 
\mathbf{p}^{(k)}) = Q(\mathbf{u}^{(k)}, \mathbf{u}^{(k)}; \mathbf{p}^{(k)}) 
- \frac{\|\mathbf{r}^{(k)}\|_2^4}{2(\mathbf{r}^{(k)})^\top 
\mathbf{Q}^{(k)}(\mathbf{p}^{(k)}) \mathbf{r}^{(k)}}.
\]
Since $(\mathbf{r}^{(k)})^\top \mathbf{Q}^{(k)} \mathbf{r}^{(k)} \leq 
\bar{\mu}\|\mathbf{r}^{(k)}\|_2^2$ (Lemma~\ref{lem:Q_bound}), we obtain
\[
Q(\mathbf{u}^{(k+1)}, \mathbf{u}^{(k)}; \mathbf{p}^{(k)}) \leq 
Q(\mathbf{u}^{(k)}, \mathbf{u}^{(k)}; \mathbf{p}^{(k)}) - 
\frac{\|\mathbf{r}^{(k)}\|_2^2}{2\bar{\mu}}.
\]
The result then follows from the majorization property 
$\mathcal{J}_{\varepsilon,\lambda}(\mathbf{u}^{(k+1)}; \mathbf{p}^{(k)}) 
\leq Q(\mathbf{u}^{(k+1)}, \mathbf{u}^{(k)}; \mathbf{p}^{(k)})$, the 
tangency condition $Q(\mathbf{u}^{(k)}, \mathbf{u}^{(k)}; \mathbf{p}^{(k)}) 
= \mathcal{J}_{\varepsilon,\lambda}(\mathbf{u}^{(k)}; \mathbf{p}^{(k)})$, 
and the tangency of gradients 
$\nabla_{\mathbf{u}} Q(\mathbf{u}^{(k)}, \mathbf{u}^{(k)}; \mathbf{p}^{(k)}) 
= \nabla_{\mathbf{u}} \mathcal{J}_{\varepsilon,\lambda}(\mathbf{u}^{(k)}; 
\mathbf{p}^{(k)}) = \mathbf{r}^{(k)}$.
\end{proof}

\begin{lemma}[Lipschitz continuity of $\nabla_{\mathbf{p}} \mathcal{J}$]
\label{lem:lipschitz_grad}
For fixed $\mathbf{u}$ with $\|\mathbf{u}\|_2 \leq R_u$, the gradient 
$\nabla_{\mathbf{p}} \mathcal{J}_{\epsilon,\lambda}(\mathbf{u}, \mathbf{p})$ 
is Lipschitz continuous in $\mathbf{p}$ with constant
\begin{equation}
L_{\mathbf{p}} = C_H R_u^2 + C_{HH} R_u (C_{\mathbf{H}} R_u + C_{\mathbf{b}}).
\label{eq:Lp}
\end{equation}
\end{lemma}

\begin{proof}
The regularization term does not depend on $\mathbf{p}$, so it suffices to 
consider $\mathcal{J}_{\mathrm{data}}(\mathbf{u}, \mathbf{p}) = \frac{1}{2}
\|\mathbf{H}(\mathbf{p})\mathbf{u} - \mathbf{b}\|_2^2$. The $i$-th gradient 
component is
\[
\left[\nabla_{\mathbf{p}} \mathcal{J}_{\mathrm{data}}\right]_i = 
\left(\frac{\partial \mathbf{H}}{\partial p_i}\mathbf{u}\right)^\top 
(\mathbf{H}(\mathbf{p})\mathbf{u} - \mathbf{b}).
\]
For $\mathbf{p}_1, \mathbf{p}_2$, adding and subtracting a cross term 
and applying Cauchy-Schwarz gives
\[
\left|\left[\nabla_{\mathbf{p}} \mathcal{J}_{\mathrm{data}}(\mathbf{u},
\mathbf{p}_1) - \nabla_{\mathbf{p}} \mathcal{J}_{\mathrm{data}}(\mathbf{u},
\mathbf{p}_2)\right]_i\right| \leq T_1 + T_2,
\]
where $T_1 = \left\|\frac{\partial \mathbf{H}(\mathbf{p}_1)}{\partial p_i}
\mathbf{u}\right\|_2 \|\mathbf{r}(\mathbf{p}_1) - \mathbf{r}(\mathbf{p}_2)
\|_2$ and $T_2 = \left\|\left(\frac{\partial \mathbf{H}(\mathbf{p}_1)}
{\partial p_i} - \frac{\partial \mathbf{H}(\mathbf{p}_2)}{\partial p_i}
\right)\mathbf{u}\right\|_2 \|\mathbf{r}(\mathbf{p}_2)\|_2$.
By Assumption~\ref{assump:smoothness} and the mean value theorem:
\[
\|\mathbf{r}(\mathbf{p}_1) - \mathbf{r}(\mathbf{p}_2)\|_2 \leq C_H R_u 
\|\mathbf{p}_1 - \mathbf{p}_2\|_2, \quad 
\left\|\left(\frac{\partial \mathbf{H}(\mathbf{p}_1)}{\partial p_i} - 
\frac{\partial \mathbf{H}(\mathbf{p}_2)}{\partial p_i}\right)\mathbf{u}
\right\|_2 \leq C_{HH} R_u \|\mathbf{p}_1 - \mathbf{p}_2\|_2,
\]
and by Assumption~\ref{assump:bounded}, $\|\mathbf{r}(\mathbf{p}_2)\|_2 \leq 
C_{\mathbf{H}} R_u + C_{\mathbf{b}}$. Substituting and taking the Euclidean 
norm over all $i$ yields \eqref{eq:Lp}.
\end{proof}

\begin{lemma}[Gauss-Newton descent direction]
\label{lem:gn_descent}
The Gauss-Newton direction $\mathbf{d}^{(k)} = -(\mathbf{J}_{\mathbf{p}}^{(k)})^{-1}
\mathbf{g}^{(k)}$ satisfies, whenever $\mathbf{g}^{(k)} \neq \mathbf{0}$:
\begin{equation}
\langle \mathbf{g}^{(k)}, \mathbf{d}^{(k)} \rangle \leq
-\frac{1}{\gamma_{\max}} \|\mathbf{g}^{(k)}\|_2^2 < 0.
\end{equation}
\end{lemma}

\begin{proof}
By Assumption~\ref{assump:gn} and the spectral theorem for symmetric
positive definite matrices:
\[
\langle \mathbf{g}^{(k)}, \mathbf{d}^{(k)} \rangle =
-(\mathbf{g}^{(k)})^\top (\mathbf{J}_{\mathbf{p}}^{(k)})^{-1} \mathbf{g}^{(k)}
\leq -\frac{\|\mathbf{g}^{(k)}\|_2^2}{\lambda_{\max}
(\mathbf{J}_{\mathbf{p}}^{(k)})} \leq -\frac{\|\mathbf{g}^{(k)}\|_2^2}
{\gamma_{\max}}.
\]
\end{proof}

\begin{lemma}[Gauss-Newton bounded norm]
\label{lem:gn_norm}
The Gauss-Newton direction $\mathbf{d}^{(k)} = -(\mathbf{J}_{\mathbf{p}}^{(k)})^{-1}
\mathbf{g}^{(k)}$ satisfies:
\begin{equation}
\|\mathbf{d}^{(k)}\|_2 \leq \frac{1}{\gamma_{\min}} \|\mathbf{g}^{(k)}\|_2.
\end{equation}
\end{lemma}

\begin{proof}
By Assumption~\ref{assump:gn}:
\[
\|\mathbf{d}^{(k)}\|_2 \leq \|(\mathbf{J}_{\mathbf{p}}^{(k)})^{-1}\|_2
\|\mathbf{g}^{(k)}\|_2 = \frac{\|\mathbf{g}^{(k)}\|_2}
{\lambda_{\min}(\mathbf{J}_{\mathbf{p}}^{(k)})} \leq
\frac{\|\mathbf{g}^{(k)}\|_2}{\gamma_{\min}}.
\]
\end{proof}

\begin{lemma}[Lower bound on backtracking step size]
\label{lem:step_size}
Under backtracking line search with initial step $\bar{\alpha} > 0$, 
reduction factor $\beta \in (0,1)$, and Armijo constant $c_1 \in (0,1)$, 
the accepted step size satisfies $\alpha^{(k)} \geq \alpha_{\min}$, where
\begin{equation}
\alpha_{\min} := \min\left\{\bar{\alpha},\; 
\frac{2\beta(1-c_1)\gamma_{\min}^2}{L_{\mathbf{p}}\gamma_{\max}}\right\} > 0.
\end{equation}
\end{lemma}

\begin{proof}
By Taylor expansion and Lemma~\ref{lem:lipschitz_grad}:
\[
\mathcal{J}(\mathbf{u}^{(k)}, \mathbf{p} + \alpha \mathbf{d}^{(k)}) \leq 
\mathcal{J}(\mathbf{u}^{(k)}, \mathbf{p}) + \alpha \langle \mathbf{g}^{(k)}, 
\mathbf{d}^{(k)} \rangle + \frac{L_{\mathbf{p}} \alpha^2}{2\gamma_{\min}^2} 
\|\mathbf{g}^{(k)}\|_2^2,
\]
where we used Lemma~\ref{lem:gn_norm} to bound $\|\mathbf{d}^{(k)}\|_2$. 
If $\alpha$ fails the Armijo condition, comparing with 
$c_1 \alpha \langle \mathbf{g}^{(k)}, \mathbf{d}^{(k)} \rangle$ and using 
Lemma~\ref{lem:gn_descent} gives $\alpha > 2(1-c_1)\gamma_{\min}^2 / 
(L_{\mathbf{p}}\gamma_{\max})$. Since backtracking reduces $\alpha$ by 
factor $\beta$ before accepting, the accepted step satisfies the stated bound.
\end{proof}

\begin{lemma}[Descent from parameter update]
\label{lem:p_descent}
One Gauss-Newton iteration with backtracking line search produces:
\begin{equation}
\mathcal{J}_{\epsilon,\lambda}(\mathbf{u}^{(k)}, \mathbf{p}^{(k)}) \leq 
\mathcal{J}_{\epsilon,\lambda}(\mathbf{u}^{(k)}, \mathbf{p}^{(k-1)}) - 
\frac{c_1 \alpha_{\min}}{\gamma_{\max}} \|\mathbf{g}^{(k)}\|_2^2.
\end{equation}
\end{lemma}

\begin{proof}
By the Armijo condition and Lemmas~\ref{lem:gn_descent} 
and~\ref{lem:step_size}:
\[
\mathcal{J}_{\epsilon,\lambda}(\mathbf{u}^{(k)}, \mathbf{p}^{(k)}) \leq 
\mathcal{J}_{\epsilon,\lambda}(\mathbf{u}^{(k)}, \mathbf{p}^{(k-1)}) + 
c_1 \alpha^{(k)} \langle \mathbf{g}^{(k)}, \mathbf{d}^{(k)} \rangle \leq 
\mathcal{J}_{\epsilon,\lambda}(\mathbf{u}^{(k)}, \mathbf{p}^{(k-1)}) - 
\frac{c_1 \alpha_{\min}}{\gamma_{\max}} \|\mathbf{g}^{(k)}\|_2^2. 
\]
\end{proof}

\subsection{Main Proof}

\begin{proof}[Proof of Theorem~\ref{thm:convergence}]
Define $\kappa_u := (2\bar{\mu})^{-1}$ and $\kappa_p := c_1\alpha_{\min}
\gamma_{\max}^{-1}$, both positive by Lemmas~\ref{lem:Q_bound} 
and~\ref{lem:step_size}.

\textbf{Step 1: Per-iteration descent.}
Applying Lemma~\ref{lem:u_descent} to the image update and 
Lemma~\ref{lem:p_descent} to the parameter update at each outer 
iteration $k$:
\begin{equation}
\mathcal{J}_{\epsilon,\lambda}(\mathbf{u}^{(k)}, \mathbf{p}^{(k)}) \leq 
\mathcal{J}_{\epsilon,\lambda}(\mathbf{u}^{(k-1)}, \mathbf{p}^{(k-1)}) 
- \kappa_u \|\nabla_{\mathbf{u}} \mathcal{J}(\mathbf{u}^{(k-1)}, 
\mathbf{p}^{(k-1)})\|_2^2 - \kappa_p \|\mathbf{g}^{(k)}\|_2^2.
\label{eq:global_descent}
\end{equation}

\textbf{Step 2: Summability.}
Since $\mathcal{J}_{\epsilon,\lambda} \geq 0$, telescoping 
\eqref{eq:global_descent} from $k = 1$ to $K$ gives
\[
\kappa_u \sum_{k=0}^{K-1} \|\nabla_{\mathbf{u}} \mathcal{J}(\mathbf{u}^{(k)}, 
\mathbf{p}^{(k)})\|_2^2 + \kappa_p \sum_{k=1}^{K} \|\mathbf{g}^{(k)}\|_2^2 
\leq \mathcal{J}_{\epsilon,\lambda}(\mathbf{u}^{(0)}, \mathbf{p}^{(0)}) 
< \infty.
\]
Taking $K \to \infty$:
\begin{equation}
\sum_{k=0}^{\infty} \|\nabla_{\mathbf{u}} \mathcal{J}(\mathbf{u}^{(k)}, 
\mathbf{p}^{(k)})\|_2^2 < \infty \quad \text{and} \quad 
\sum_{k=1}^{\infty} \|\mathbf{g}^{(k)}\|_2^2 < \infty.
\label{eq:summability}
\end{equation}

\textbf{Step 3: $\nabla_{\mathbf{u}}\mathcal{J} \to 0$.}
A non-negative sequence whose terms are summable must converge to zero, 
so \eqref{eq:summability} immediately gives
\[
\lim_{k \to \infty} \|\nabla_{\mathbf{u}} \mathcal{J}_{\epsilon,\lambda}
(\mathbf{u}^{(k)}, \mathbf{p}^{(k)})\|_2 = 0.
\]

\textbf{Step 4: $\nabla_{\mathbf{p}}\mathcal{J} \to 0$.}
Recall $\mathbf{g}^{(k)} = \nabla_{\mathbf{p}} \mathcal{J}(\mathbf{u}^{(k)}, 
\mathbf{p}^{(k-1)})$. Since $\|\mathbf{g}^{(k)}\|_2 \to 0$ from 
\eqref{eq:summability}, it remains to transfer this to 
$\nabla_{\mathbf{p}} \mathcal{J}(\mathbf{u}^{(k)}, \mathbf{p}^{(k)})$. 
By Lemma~\ref{lem:gn_norm} and the update rule:
\[
\|\mathbf{p}^{(k)} - \mathbf{p}^{(k-1)}\|_2 \leq 
\frac{\bar{\alpha}}{\gamma_{\min}} \|\mathbf{g}^{(k)}\|_2.
\]
Applying Lemma~\ref{lem:lipschitz_grad} and the triangle inequality:
\[
\|\nabla_{\mathbf{p}} \mathcal{J}(\mathbf{u}^{(k)}, \mathbf{p}^{(k)})\|_2 
\leq \left(\frac{L_{\mathbf{p}} \bar{\alpha}}{\gamma_{\min}} + 1\right) 
\|\mathbf{g}^{(k)}\|_2 \to 0. 
\]
\end{proof}

\subsection{Extension to Multiple Inner Iterations}

\begin{theorem}[Convergence with multiple inner iterations]
\label{thm:convergence_multi}
Under Assumptions~\ref{assump:smoothness}--\ref{assump:gn}, if RMM-GKS 
performs $m_k \geq 1$ inner iterations and the parameter update performs 
$\mathrm{maxiter}_p \geq 1$ inner Gauss-Newton iterations at each outer 
iteration $k$, then
\[
\lim_{k \to \infty} \|\nabla_{\mathbf{u}} \mathcal{J}_{\epsilon,\lambda}
(\mathbf{u}^{(k)}, \mathbf{p}^{(k)})\|_2 = 0 \quad \text{and} \quad 
\lim_{k \to \infty} \|\nabla_{\mathbf{p}} \mathcal{J}_{\epsilon,\lambda}
(\mathbf{u}^{(k)}, \mathbf{p}^{(k)})\|_2 = 0.
\]
\end{theorem}

\begin{proof}
Applying Lemma~\ref{lem:u_descent} to each of the $m_k$ inner image 
iterations and telescoping:
\begin{equation}
\mathcal{J}(\mathbf{u}^{(k+1)}, \mathbf{p}^{(k)}) \leq 
\mathcal{J}(\mathbf{u}^{(k)}, \mathbf{p}^{(k)}) - \kappa_u 
\sum_{i=0}^{m_k-1} \|\nabla_{\mathbf{u}} \mathcal{J}(\mathbf{u}^{(k,i)}, 
\mathbf{p}^{(k)})\|_2^2.
\label{eq:u_multi}
\end{equation}
Applying Lemma~\ref{lem:p_descent} to each of the $\mathrm{maxiter}_p$ 
inner parameter iterations and telescoping:
\begin{equation}
\mathcal{J}(\mathbf{u}^{(k+1)}, \mathbf{p}^{(k+1)}) \leq 
\mathcal{J}(\mathbf{u}^{(k+1)}, \mathbf{p}^{(k)}) - \kappa_p 
\sum_{\ell=0}^{\mathrm{maxiter}_p - 1} \|\mathbf{g}^{(k+1,\ell)}\|_2^2.
\label{eq:p_multi}
\end{equation}
Combining \eqref{eq:u_multi} and \eqref{eq:p_multi} and summing over 
$k = 0, \ldots, K-1$:
\[
\kappa_u \sum_{k=0}^{K-1} \sum_{i=0}^{m_k-1} \|\nabla_{\mathbf{u}} 
\mathcal{J}(\mathbf{u}^{(k,i)}, \mathbf{p}^{(k)})\|_2^2 + \kappa_p 
\sum_{k=0}^{K-1} \sum_{\ell=0}^{\mathrm{maxiter}_p-1} 
\|\mathbf{g}^{(k+1,\ell)}\|_2^2 \leq \mathcal{J}(\mathbf{u}^{(0)}, 
\mathbf{p}^{(0)}) < \infty.
\]
Taking $K \to \infty$, both double sums are finite. Since 
$\mathbf{u}^{(k)} = \mathbf{u}^{(k,0)}$ appears in the first sum, 
$\|\nabla_{\mathbf{u}} \mathcal{J}(\mathbf{u}^{(k)}, \mathbf{p}^{(k)})\|_2 
\to 0$ by summability of non-negative terms. For the $\mathbf{p}$-gradient, 
since $\|\mathbf{g}^{(k+1, \mathrm{maxiter}_p - 1)}\|_2 \to 0$ by summability, 
the same Lipschitz argument as in Step 4 of the main proof gives
\[
\|\nabla_{\mathbf{p}} \mathcal{J}(\mathbf{u}^{(k+1)}, \mathbf{p}^{(k+1)})\|_2 
\leq \left(\frac{L_{\mathbf{p}}\bar{\alpha}}{\gamma_{\min}} + 1\right) 
\|\mathbf{g}^{(k+1, \mathrm{maxiter}_p - 1)}\|_2 \to 0. 
\]
\end{proof}

\section{Supplementary Material}
Section~\ref{sec:supp_algorithms} collects algorithms that were omitted 
from the main text to save space. 
Section~\ref{sec:supp_inner_iters} studies the effect of inner iteration 
count on streaming reconstruction quality. 

\subsection{Supporting Algorithms}
\label{sec:supp_algorithms}

\subsubsection{MM-GKS Algorithm}
\label{sec:supp_mmgks}

Algorithm~\ref{alg:mmgks} gives the full MM-GKS procedure 
\cite{mm-gks,lanza2015generalized} for solving the regularized linear 
subproblem
\[
\min_{\mathbf{u}} \frac{1}{2}\|\mathbf{H}\mathbf{u} - \mathbf{b}\|_2^2 
+ \lambda \sum_{j=1}^n \phi_\varepsilon((\boldsymbol{\Psi}\mathbf{u})_j),
\]
given a fixed forward operator $\mathbf{H}$ and regularization 
parameter $\lambda$. This is the core inner solver called by all 
variants of NL-RMM-GKS.

\begin{algorithm}
\caption{MM-GKS}
\label{alg:mmgks}
\begin{algorithmic}[1]
\REQUIRE $\mathbf{H},\,\boldsymbol{\Psi},\,\mathbf{b},\,\mathbf{u}^{(0)},\,\epsilon$
\ENSURE Approximate solution $\mathbf{u}^{(k+1)}$

\STATE Generate initial basis $\mathbf{V}_\ell$ with $\mathbf{V}_\ell^\top \mathbf{V}_\ell = \mathbf{I}$

\FOR{$k = 0,1,2,\ldots$ until convergence}

    \STATE $\mathbf{s}^{(k)} = \boldsymbol{\Psi}\mathbf{u}^{(k)}$ 
    \STATE $\mathbf{w}^{(k)}_\epsilon = ((\mathbf{s}^{(k)})^2 + \epsilon^2)^{-1/2}$
    \STATE $\mathbf{P}^{(k)}_\epsilon = \mathrm{diag}(\mathbf{w}^{(k)}_\epsilon)^{1/2}$

    \STATE $\mathbf{H}\mathbf{V}_{\ell+k} = \mathbf{Q}_H \mathbf{R}_H$
    \STATE $\mathbf{P}^{(k)}_\epsilon\boldsymbol{\Psi}\mathbf{V}_{\ell+k}
            = \mathbf{Q}_\Psi \mathbf{R}_\Psi$

    \STATE Select $\lambda^{(k)}$ by GCV

    \STATE 
    $\mathbf{z}^{(k+1)}
        = \arg\min_{\mathbf{z}}
          \left\|
          \begin{bmatrix}
            \mathbf{R}_H \\[2pt]
            \sqrt{\lambda^{(k)}}\,\mathbf{R}_\Psi
          \end{bmatrix}
          \mathbf{z}
          -
          \begin{bmatrix}
            \mathbf{Q}_H^\top\mathbf{b} \\[2pt]
            \mathbf{0}
          \end{bmatrix}
          \right\|_2^2$

    \STATE $\mathbf{u}^{(k+1)} = \mathbf{V}_{\ell+k}\mathbf{z}^{(k+1)}$

    \STATE $\mathbf{r}^{(k+1)}
           = \mathbf{H}^\top(\mathbf{H}\mathbf{u}^{(k+1)} - \mathbf{b})
           + \lambda^{(k)} \boldsymbol{\Psi}^\top (\mathbf{P}^{(k)}_\epsilon)^2 
             \boldsymbol{\Psi}\mathbf{u}^{(k+1)}$

    \STATE $\mathbf{r}^{(k+1)}
           \leftarrow
           \mathbf{r}^{(k+1)} - 
           \mathbf{V}_{\ell+k}\mathbf{V}_{\ell+k}^\top \mathbf{r}^{(k+1)}$

    \STATE $\mathbf{v}_{\text{new}}
           = \mathbf{r}^{(k+1)} / \|\mathbf{r}^{(k+1)}\|_2$

    \STATE $\mathbf{V}_{\ell+k+1} = [\mathbf{V}_{\ell+k}, \mathbf{v}_{\text{new}}]$

\ENDFOR
\end{algorithmic}
\end{algorithm}

\subsubsection{Enlarge and Compress Subroutines}
\label{sec:supp_enlarge_compress}

Algorithms~\ref{alg:enlarge} and~\ref{alg:compress} implement the 
subspace management subroutines used by RMM-GKS and all its extensions. 
\textsc{Enlarge} expands the current basis $\mathbf{V}_{k_{\min}}$ by 
appending new Golub-Kahan vectors until the basis reaches size $k_{\max}$. 
\textsc{Compress} reduces the enlarged basis back to size $k_{\min}$ via 
a truncated SVD, retaining the $k_{\min}$ directions of greatest variance 
to carry forward to the next iteration.
\begin{algorithm}
\caption{Enlarge}
\label{alg:enlarge}
\begin{algorithmic}[1]

\REQUIRE $\mathbf{H},\,\boldsymbol{\Psi},\,\mathbf{V}_{k_{\min}},\,\mathbf{d},\,
          \mathbf{u}^{(0)},\,\epsilon,\,s,\,\text{tol}_1,\,\lambda_{\mathrm{fix}}$
\ENSURE $\mathbf{u}^{(k_{\max})},\,\lambda^{(k_{\max})},\,
        \mathbf{V}_{k_{\max}},\,\mathbf{R}_H,\,\mathbf{R}_\Psi$

\FOR{$k = 0,\ldots,s-1$}

    \STATE $\mathbf{s}^{(k)} = \boldsymbol{\Psi}\mathbf{u}^{(k)}$
    \STATE $\mathbf{w}^{(k)}_\epsilon = ((\mathbf{s}^{(k)})^2+\epsilon^2)^{-1/2}$
    \STATE $\mathbf{P}^{(k)}_\epsilon = \mathrm{diag}(\mathbf{w}^{(k)}_\epsilon)^{1/2}$

    \STATE $\mathbf{H}\mathbf{V}_{k_{\min}+k} = \mathbf{Q}_H \mathbf{R}_H$
    \STATE $\mathbf{P}^{(k)}_\epsilon\boldsymbol{\Psi}\mathbf{V}_{k_{\min}+k}
            = \mathbf{Q}_\Psi \mathbf{R}_\Psi$

    \IF{$\lambda_{\mathrm{fix}}$ given}
        \STATE $\lambda^{(k)} = \lambda_{\mathrm{fix}}$
    \ELSE
        \STATE $\lambda^{(k)} = \arg\min_\lambda \Theta(\lambda)$
    \ENDIF

    \STATE $\mathbf{z}^{(k+1)}
           = (\mathbf{R}_H^\top\mathbf{R}_H +
              \lambda^{(k)}\mathbf{R}_\Psi^\top\mathbf{R}_\Psi)^{-1}
             \mathbf{R}_H^\top\mathbf{Q}_H^\top\mathbf{d}$

    \STATE $\mathbf{u}^{(k+1)} = \mathbf{V}_{k_{\min}+k}\mathbf{z}^{(k+1)}$

    \STATE $\mathbf{r}^{(k+1)}
           = \mathbf{H}^\top(\mathbf{H}\mathbf{u}^{(k+1)} - \mathbf{d})
           + \lambda^{(k)}\boldsymbol{\Psi}^\top(\mathbf{P}^{(k)}_\epsilon)^2
             \boldsymbol{\Psi}\mathbf{u}^{(k+1)}$

    \STATE $\mathbf{r}^{(k+1)}
            \leftarrow \mathbf{r}^{(k+1)}
            - \mathbf{V}_{k_{\min}+k}\mathbf{V}_{k_{\min}+k}^\top\mathbf{r}^{(k+1)}$

    \STATE $\mathbf{v}_{\text{new}}
            = \mathbf{r}^{(k+1)} / \|\mathbf{r}^{(k+1)}\|_2$

    \STATE $\mathbf{V}_{k_{\min}+k+1}
            = [\mathbf{V}_{k_{\min}+k},\,\mathbf{v}_{\text{new}}]$

    \IF{$\|\mathbf{u}^{(k+1)} - \mathbf{u}^{(k)}\|_2
         / \|\mathbf{u}^{(k)}\|_2 \le \text{tol}_1$}
        \STATE break
    \ENDIF

\ENDFOR

\STATE $\mathbf{u}^{(k_{\max})} = \mathbf{u}^{(k+1)}$
\STATE $\mathbf{V}_{k_{\max}} = \mathbf{V}_{k_{\min}+k+1}$

\end{algorithmic}
\end{algorithm}

\begin{algorithm}
\caption{Compress}
\label{alg:compress}
\begin{algorithmic}[1]

\REQUIRE $\mathbf{V}_{k_{\max}},\,\mathbf{R}_H,\,\mathbf{R}_\Psi,\,
          \mathbf{d},\,\mathbf{u},\,\mathbf{Q}_H,\,k_{\min},\,\lambda$
\ENSURE $\mathbf{V}_{k_{\min}}$

\STATE $W = \chi(\mathbf{R}_H,\mathbf{R}_\Psi,\mathbf{Q}_H,\mathbf{d},\lambda)$

\STATE $\widetilde{\mathbf{V}} = \mathbf{V}_{k_{\max}} W$

\STATE $\mathbf{u}_{\text{new}}
        = \mathbf{u} - \widetilde{\mathbf{V}}\widetilde{\mathbf{V}}^\top\mathbf{u}$

\STATE $\mathbf{u}_{\text{new}}
        = \mathbf{u}_{\text{new}} / \|\mathbf{u}_{\text{new}}\|_2$

\STATE $\mathbf{V}_{k_{\min}}
        = [\widetilde{\mathbf{V}},\,\mathbf{u}_{\text{new}}]$

\end{algorithmic}
\end{algorithm}

\subsubsection{UPDATE-PARAM-VarPro}
\label{sec:supp_update_varpro}

Algorithm~\ref{alg:update_varpro} implements the variable projection 
parameter update used in the VarPro realization of NL-RMM-GKS 
(Section~4.2 of the main paper). It exploits the separable structure 
of the VarPro objective to compute a Gauss-Newton step directly in 
the parameter space, treating the image as implicitly defined by the 
current parameter estimate.

\begin{algorithm}
\caption{UPDATE-PARAM-VarPro}
\label{alg:update_varpro}
\begin{algorithmic}[1]
\REQUIRE $\mathbf{H}(\cdot)$, $\boldsymbol{\Psi}$, $\mathbf{b}$, 
$\mathbf{p}^{(k)}$, $\mathbf{V}_{k_{\min}}^{(k)}$, $\hat{\lambda}$, 
$\mathrm{maxiter}_p$, $c_1$, $\beta$, $\bar{\alpha}$
\ENSURE $\mathbf{p}^{(k+1)}$

\FOR{$\ell = 0, 1, \ldots, \mathrm{maxiter}_p - 1$}

    \STATE Form $\mathbf{H}_\ell = \mathbf{H}(\mathbf{p}^{(k,\ell)})$ 
    and reduced basis $\hat{\mathbf{A}}_\ell = \mathbf{H}_\ell 
    \mathbf{V}_{k_{\min}}^{(k)}$

    \STATE Compute VarPro residual:
    \[
        \mathbf{y}^{(k,\ell)} = \arg\min_{\mathbf{y}} 
        \left\|\begin{bmatrix}\hat{\mathbf{A}}_\ell \\ 
        \sqrt{\hat{\lambda}}\hat{\mathbf{B}}_\ell\end{bmatrix}\mathbf{y} 
        - \begin{bmatrix}\mathbf{b} \\ 
        \mathbf{0}\end{bmatrix}\right\|_2^2, \quad 
        \hat{\mathbf{u}}^{(k,\ell)} = \mathbf{V}_{k_{\min}}^{(k)} 
        \mathbf{y}^{(k,\ell)}
    \]

    \STATE Compute VarPro gradient via implicit differentiation 
    \cite{golub1973differentiation,golub2003separable}:
    \[
        \mathbf{g}^{(k,\ell)} = \nabla_{\mathbf{p}} 
        \mathcal{J}_{\varepsilon,\hat{\lambda}}(\hat{\mathbf{u}}^{(k,\ell)}, 
        \mathbf{p}^{(k,\ell)})
    \]

    \STATE Form damped Gauss-Newton matrix 
    $\mathbf{J}^{(k,\ell)}_{\mathbf{p}}$ and compute direction:
    \[
        \mathbf{d}^{(k,\ell)} = -(\mathbf{J}^{(k,\ell)}_{\mathbf{p}})^{-1} 
        \mathbf{g}^{(k,\ell)}
    \]

    \STATE Backtracking line search with Armijo condition ($c_1$, 
    $\beta$, $\bar{\alpha}$) to find $\alpha^{(k,\ell)}$

    \STATE $\mathbf{p}^{(k,\ell+1)} = \mathbf{p}^{(k,\ell)} + 
    \alpha^{(k,\ell)} \mathbf{d}^{(k,\ell)}$

\ENDFOR

\STATE $\mathbf{p}^{(k+1)} = \mathbf{p}^{(k,\mathrm{maxiter}_p)}$
\end{algorithmic}
\end{algorithm}

\subsection{Effect of Inner Iteration Count on Streaming Performance}
\label{sec:supp_inner_iters}

We investigate how the number of inner MM-GKS iterations per block 
affects reconstruction quality and runtime in the streaming setting, 
using the Test~1 setup from Section~6.2 of the main paper 
(static Shepp-Logan CT, $N = 4$ blocks, $p_{\mathrm{true}} = 0.2421^\circ$). 
We vary the inner iteration count $m \in \{5, 10, 15, 20\}$ and report 
final RRE, parameter error, and wall-clock time.

\paragraph{Results.}
Figure~\ref{fig:supp_inner_iters} shows that reconstruction quality 
improves as $m$ increases from 5 to 10, with diminishing returns 
beyond $m = 10$. Specifically, RRE decreases from 0.2103 at $m=5$ 
to 0.1453 at $m=10$, and only marginally further to 0.1388 at $m=20$, 
while runtime increases roughly linearly with $m$. Parameter estimation 
is largely insensitive to $m$ across the tested range. These results 
support the choice of $m = 10$ inner iterations used throughout the 
main paper as a good balance between accuracy and computational cost. 
Table~\ref{tab:supp_inner_iters} gives the full numerical summary.

\begin{figure}[htbp]
\centering
\subfloat[RRE vs. outer iteration for varying inner iterations (5, 10, 15, 20)]{
    \includegraphics[width=.9\textwidth]{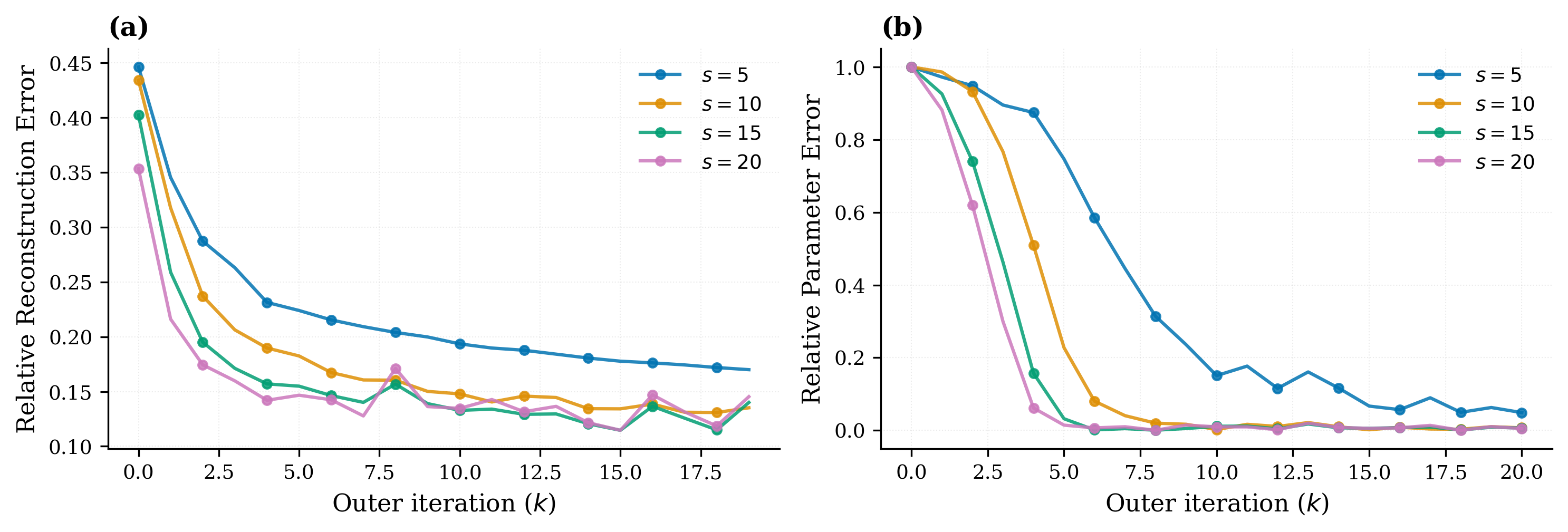}
    \label{fig:exp_1b_convergence_plots}
}\\[.1em]
\centering
\subfloat[Visual reconstructions comparing ground truth with results from 5, 10, 15, 20 inner iterations]{
    \includegraphics[width = .9\textwidth]{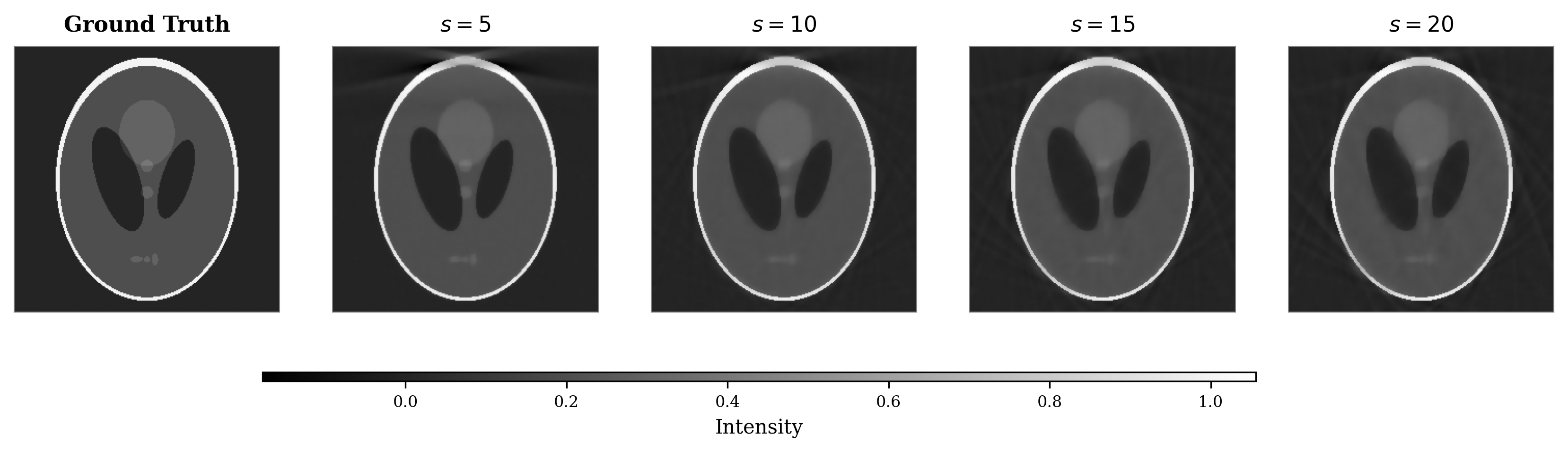}
    \label{fig:exp_1b_reconstruction_plots}
}
\caption{Effect of inner iteration count $m$ on streaming NL-RMM-GKS ($N=4$ blocks, Test~1 setup). 
(a) shows diminishing returns beyond 10 inner iterations, with all configurations converging to similar final RRE. 
(b) confirms that visual quality is comparable across settings, suggesting 10 iterations provides the best speed-accuracy balance.}
\label{fig:supp_inner_iters}
\end{figure}

\begin{table}[htbp]
\centering
\caption{Effect of inner iteration count on s-NL-RMM-GKS 
($N=4$ blocks, Test~1 setup, $p_{\mathrm{true}} = 0.2421^\circ$).}
\label{tab:supp_inner_iters}
\begin{tabular}{lcccccc}
\toprule
Inner Iter & Time (min) & Peak Mem. (GB) & Final RRE& Param. Err & $p^*$ (\degree)\\
\midrule
5 & 0.22& 0.088& 0.1699& 0.0490&0.2302\\
10 & 0.41& 0.118& 0.1351& 0.0071&0.2404\\
15 & 0.67& 0.152& 0.1400& 0.0066&0.2405\\
20 & 0.89& 0.188& 0.1453& 0.0053&0.2408\\
\bottomrule
\end{tabular}
\end{table}

\end{document}